\documentclass[a4paper,10pt]{article}
\usepackage[utf8x]{inputenc}
\usepackage[T1]{fontenc}
\usepackage[english]{babel}

\usepackage{amsmath}
\usepackage{amsfonts}
\usepackage{amsthm}
\usepackage{amssymb}
\usepackage{enumitem}
\usepackage{appendix}
\usepackage{geometry}
\usepackage{eurosym}
\usepackage{hyperref}
\hypersetup{
    colorlinks,
    citecolor=blue,
    filecolor=blue,
    linkcolor=blue,
    urlcolor=blue
}
\usepackage[textsize=tiny]{todonotes}

\usepackage{rotating}
\usepackage[round]{natbib}
\usepackage{empheq}

\def\1{{\bf 1}}

\usepackage[textsize=tiny]{todonotes}

\definecolor{cadmiumgreen}{rgb}{0.0, 0.42, 0.24}

\def \V{\mathbb{V}}

\newcommand{\EE}{\mathbb{E}}
\newcommand{\NN}{\mathbb{N}}

\newcommand{\RR}{\mathbb{R}}

\newcommand{\D}{\mathrm{d}}

\newcommand{\Ff}{\mathcal{F}}

\newcommand{\Xx}{\mathcal{X}}

\definecolor{ocean}{rgb}{0,0.1,0.6}

\definecolor{imperialGreen}{RGB}{2,137,59}
\definecolor{imperialBlue}{RGB}{0, 62, 116}
\definecolor{imperialBrick}{RGB}{165,25,0}
\definecolor{imperialProcess}{RGB}{0,133,202}

\mathtoolsset{showonlyrefs=true}

\def \V{\mathbb{V}}

\newtheorem{theorem}{Theorem}[section]

\theoremstyle{definition}
\newtheorem{definition}[theorem]{Definition}
\newtheorem{remark}[theorem]{Remark}
\newtheorem{assumption}[theorem]{Assumption}
\theoremstyle{plain}

\numberwithin{equation}{section}

\newcommand{\hs}{\vspace{3mm}}
\graphicspath{{./figs/}}

\begin{document}

\title{Continuous-time persuasion by filtering}
\author{René Aïd\thanks{Paris-Dauphine University, PSL Research University, rene.aid@dauphine.psl.eu.}\quad
Ofelia Bonesini\thanks{London School of Economics and Political Sciences, Department of Mathematics, o.bonesini@lse.ac.uk.}\quad
Giorgia Callegaro\thanks{Universit\`a degli Studi di  Padova, Dipartimento di Matematica ``Tullio Levi-Civita'', gcallega@math.unipd.it.}\quad
Luciano Campi\thanks{Universit\`a degli Studi di Milano, Dipartimento di Matematica ``Federigo Enriques'', luciano.campi@unimi.it.}
}
\maketitle

\begin{abstract}
We frame dynamic persuasion in a partial observation stochastic control Leader-Follower game with an ergodic criterion. The Receiver controls the dynamics of a multidimensional unobserved state process. Information is provided to the Receiver through a  device designed by the Sender that generates the observation process. The commitment of the Sender is enforced. We develop this approach in the case where all dynamics are linear and the preferences of the Receiver are linear-quadratic. We prove a verification theorem for the existence and uniqueness of the solution of the HJB equation satisfied by the Receiver's value function. An extension to the case of persuasion of a mean field of interacting Receivers is also provided. We illustrate this approach in two applications: the provision of information to electricity consumers with a smart meter designed by an electricity producer; the information provided by carbon footprint accounting rules to companies engaged in a best-in-class emissions reduction effort. In the first application, we link the benefits of information provision to the mispricing of electricity production. In the latter, we show that even in the absence of information cost, it might be optimal for the regulator to blur information available to firms to prevent them from coordinating on a higher level of carbon footprint to reduce their cost of reaching a below average emission target.
\end{abstract}




\section{Introduction}
In conjunction with technologies and consumer preferences, information represents a powerful instrument for influencing the actions of individuals. The Cambridge Analytica scandal represents one of the most recent and potentially alarming instances of the use of information to influence a collective decision. To be more precise, Cambridge Analytica harvested millions of Facebook users' information with the objective of designing a targeted information policy to undecided voters in the 2016 US Presidential election (for a detailed account of the precise nature of the information collected from Facebook users and their information targeting strategy, see \cite{Rehman19}). If this case captured the attention of public authorities due to its magnitude, the sale campaigns on social networks that employ influencers as mediators also expose all of us to information provided by agents whose intentions are not necessarily aligned with the best interests of consumers\footnote{This is particularly relevant in light of the recent French law (2023-451) of 9 June 2023 on the regulation of the activities of social media influencers.}. Lobbyists in the pharmaceutical industry or tobacco companies provide further examples of situations where information is the natural tool used to influence the decision of a public policy maker. Thus, on one hand, it would be remiss to ignore the examples of situations where information is provided to the benefit of the influencer. On the other hand, it is important to recognise that there are alternative situations where information provision is required for the benefit of the consumer. For example, the dissemination of accurate information about the state of the pandemic was a significant factor in obtaining public acceptance of the necessary restrictions to save lives. For further insights on this last framework see  \cite{Pathak22}.

The situations above fit in the framework of signaling games and, since the seminal paper of \cite{Crawford82} on the strategic provision of information between an informed Sender and an uninformed Receiver, considerable attention has been given in the economic literature on a particular form of them, namely {\em persuasion games}. In signaling games, a Sender (she) tries to influence for her best interest the action of a Receiver (he) by sending him a message on the state of the world. Two cases can occur: either the Sender does observe the true state of Nature and she sends a message accordingly or the Sender designs a mechanism that generates messages as a function of the state that she may not even observe. The activity of influencers can be classified in the first class: they observe the true qualities of a product but they can decide to highlight certain aspects to drive the consumer into buying the product. They enter into a class of signaling games that can be referred to as {\em cheap talk} (see \cite{forges2020annals}). On the other hand, going back to our previous example, a pharmaceutical company falls into the second scenario. To obtain a license from the drug administration, the firm must design a testing protocol that produces results based on the natural variability of the sample cases. However, the company is required to design this test before the actual outcome of Nature's draw is known. The two cases differ in the fact that, in the first one, the Receiver only observe the message whereas, in the second situation, the Receiver knows the {\em message rule}. In other words, in the first case ({\em cheap talk}), the equilibrium concept is of a Nash-equilibrium in mutual best response between the Sender and the Receiver, so that, at equilibrium, the Receiver knows the conditional probability $\pi(m|\omega)$ of receiving a certain message value $m$ when the state of Nature is $\omega$.

The framework of Bayesian persuasion developed by \cite{Kamenica11} made the choice of another solution concept, namely a Stackelberg equilibrium between the Sender and the Receiver where the Sender {\em commits} to a given conditional probability $\pi(m|\omega)$ observed by the Receiver. In a sense, Bayesian persuasion can be seen as the information design counterpart of the monetary incentives in optimal contract theory of Principal-Agent's game. Bayesian persuasion has gathered considerable attention, as evidenced by the numerous citations and its application across various fields of social sciences (see \cite{Kamenica19} for an earlier survey of application papers). However, many real-life economic situations involve the dynamic provision of information: increasing sales requires constant provision of advertisements, winning investors fundings requires a flow of positive signal on the investment. As a result, the need to extend the Bayesian persuasion framework to a dynamic setting quickly became a priority on the economic research agenda of persuasion.

\medskip

In this paper, we model dynamic persuasion as a stochastic control game with partial observation and we address it using filtering theory. In our framework, the Receiver controls the dynamics of the unobserved (and possibly multidimensional) state. Information is provided to the Receiver thanks to a {\em device} that generates an observation process  which serves as the information process (also potentially multidimensional). The Sender designs the characteristics of the device, i.e., she chooses the values of its parameters. Once the device is designed and in operation, the Sender has no longer control over the trajectories of the information process. The Receiver then solves a stochastic control problem with partial information. We assume that the Sender is concerned with the stationary operation of its information provision mechanism. Therefore, we embed the Receiver with an ergodic criterion, which captures the stationary long-run behaviour of the Receiver. Furthermore, we develop this approach in the case where all dynamics are linear and the preferences of the Receiver are linear-quadratic. In this setting, the separation principle applies and the stochastic control problem of the Receiver can be solved using only the dynamics of the estimated state provided by the Kalman-Bucy filter (see \cite{Bensoussan18}, Section 9.3). Moreover, although the preferences of the Receiver are restricted to be linear-quadratic, such a restriction does not apply to the Sender, for which only the stationary law of the state variables and their conditional expectation matter. In this context, we prove a verification theorem for the existence and uniqueness of the solution of the HJB equation satisfied by the Receiver's value function. An extension to the case of persuasion of a mean field of interacting Receivers is also provided.
\medskip

This approach contributes to the research agenda proposed by \cite{Kamenica21}, where the authors call for the development of models where the structure of the information provision mechanism is constrained to better reflect socio-economic situations. Our framework introduces several advancements that extend the applicability of persuasion models to economics.

In the first place, it explicitly links dynamic persuasion, and in particular continuous-time persuasion, to filtering theory and stochastic control with partial observation. 
Surprisingly, with the exception of \cite{Escude23}, so far none of the papers dealing with dynamic persuasion, and in particular continuous-time persuasion, mentioned that the dynamics of the estimated state is given by the filtering equation. In its most general formulation, the dynamics of the estimated state in continuous time leads to the Zakai equation (see \cite{Bain09}, Section 3.5), which is a stochastic differential equation in infinite dimension. When coupled to a stochastic control problem, there are two settings that allow the infinite dimensionality of the Zakai equation to be reduced to a finite dimensional problem: the  state taking a finite number of values and the linear-quadratic case of the Kalman-Bucy filter. The literature on Bayesian persuasion has focused its analysis on the first case (with the state taking two values), which has now become so prominent to be called {\em the canonical model}, after the proposition made by \cite{Che23}. This paper develops in some detail the second case, which has not yet been treated in the literature. 

In the second place, an immediate dividend of the linear-quadratic Kalman-Bucy optimal control setting is that it allows to consider controlled states, a lacking feature in the current literature. While in the static case the distinction between action and state is not of much use, it becomes crucial in dynamic contexts opening up the possibility of modelling more complex situations. This framework makes it possible to draw on existing techniques and results from stochastic control and stochastic game theory to develop applications. Indeed, the case of a mean-field interaction of Receivers has been addressed without difficulty. 

Finally, in our framework, the Sender's strategy is defined as a device that generates the information process. 
This definition has the advantage of enforcing the Sender's commitment. As the Sender's decision is {\em static}, she is not allowed to deviate once the information mechanism starts providing messages to the Receiver.  Nevertheless, the information provision remains dynamic in the sense that the Receiver can only gradually get closer to the true value of the state, with the Sender controlling the speed at which the Receiver can identify the true value of the state. This contrasts with the literature on dynamic Bayesian persuasion, both in discrete and continuous time. The current preferred modelling is to allow the Sender to choose an information rate at each instant (see \cite{Renault13}, \cite{Golosov14}, \cite{Au15}, \cite{Renault17}, \cite{Ely17}, \cite{Ely19}, \cite{Orlov20}, \cite{Bizzotto21}, \cite{Liao21}, \cite{Ball23}, \cite{Che23}, \cite{Yao23}, \cite{Zhao24}). However, this conventional approach generally assumes that  the Receiver behaves as a short-sighted agent rather than as an agent with rational anticipations. By making the opposite choice, we ensure the Sender's commitment while allowing the Receiver to act rationally. Extending this framework to allow the Sender to dynamically alter the information rate parameter would result in a Leader-Follower stochastic differential game with partial information, which is still an open problem even in the case of a linear-quadratic setting with Gaussian noise. To the best of our knowledge, only the case of linear-quadratic games with full information has been solved via closed-form solutions (\cite{Bensoussan13}).

\medskip

Furthermore,  to illustrate the interests of this framework and prove the effectiveness of this novel approach in modelling economic situations, we provide two real-life applications. In these applications, the interests of the Sender and the Receiver may be aligned (depending of the parameters model value). Thus, in the absence of information device cost, the Sender would provide the true state of Nature. This is not necessarily the trend in the economic literature which mainly considers non-aligned interests between the Sender and the Receiver. Nevertheless, this focus on non-aligned interests is present in all signaling games: cheap talk and Bayesian persuasion share the same research question wether the Sender can use information to increase her value compared to providing no information. However, Bayesian persuasion differentiates itself with the hypothesis of Sender's commitment to a given conditional probability of a message. In this sense, the applications we present fit this idea of persuasion.

\medskip

The first application assesses the social value of smart meters for electricity consumption. The need to reduce to zero the carbon emissions of electric systems, responsible for about a third of total global greenhouse gas emissions, has led to the massive development of intermittent renewable energy sources, but also to actions on the electricity consumption side to make it more flexible and responsive to incentives (see \cite{Aid22} for an introduction to the subject and an incentive mechanism based on contract theory). It is recognised that smart meters bring social benefits by reducing the costs of measuring actual electricity consumption and detecting non-technical losses (i.e., fraudulent consumption). However, despite the existence of a literature reporting experiments showing positive impact of information on consumption reduction (see \cite{Karlin15}), doubts have been raised as to whether consumption reductions or shifts can be efficiently achieved using only information signals sent to consumers (see \cite{Pepermans14, Dato20}). Therefore, we develop a dynamic Bayesian persuasion model to assess the effect of providing consumers with information about their real-time electricity consumption. 

In our continuous-time persuasion model, the producer (Sender) incurs strictly convex costs to satisfy the total electricity demand of consumers and collects the electricity bill from consumers who pay the spot price of electricity, which is determined as an affine function of the total demand. The Sender designs a smart meter that allows consumers (Receivers) to obtain an estimate of their real-time consumption. The more efficient the smart meter, the more accurately the consumer knows his true consumption in real time. Here, the efficiency of the smart meter can be seen as the measurement frequency of the true consumption (e.g. from an annual basis to a minute basis). Of course, the efficiency of the smart meter comes at a cost, which we assume to be a strictly convex function of the device precision. The consumer has a preference for a certain level of consumption and can take costly measures to keep his consumption close to this preference. We consider two scenarios: either a single Receiver acting as a representative consumer, or a mean field of consumers interacting through the instantaneous spot price of electricity.

We provide the optimal action of the Receiver in closed loop form depending on the estimated process of its true consumption. We find that the average consumption is not affected by the accuracy of the device, regardless of its characteristics. Only the variance of consumption is reduced. Thus, our model captures the desired effect that consumption is only moved around by signals and not reduced (washing the cloths is just postponed but not canceled). Hence, the social benefit of the information provided by smart meters only comes from the reduction in the variance of consumption.  In particular, we provide the threshold above which the volatility of the consumption reduction is too low to induce social benefits. This threshold depends on only two factors: the mispricing of electricity generation costs and the potential variance reduction. We show that, when electricity generation costs are perfectly priced, there is no social benefit from developing smart meters to induce behavioural changes in consumption. This result refines the one given within a static setting in \cite{Aid23}. It also provides a theoretical argument on the provision of information to the energy industry supporting the development of electricity spot prices sent directly to consumers (see for example \cite{Wolak11}). 

\medskip

The second application focuses on the reduction of the firms carbon footprint. The signatories of the Paris Agreement have committed themselves to achieving carbon neutrality by 2050. An important tool for reaching this goal is the disclosure of the direct and indirect emissions of each company.  In Europe, this strategy is implemented in the \cite{Directive22}: in short CSRD for Corporate Sustainability Reporting Directive.  The idea is that making this information public leads to emulation fostering the competition between companies to reduce their carbon footprint. In order to explore how carbon footprint information disclosure shapes the collective behaviour of firms in an emulation context, we developed a simple model of best-in-class competition.  In this model, firms base their abatement efforts on the stringency of carbon footprint disclosure regulations, using this information to determine their own level of commitment to reducing emissions.

In our model, the Sender is the regulator who designs the guidelines and constraints for carbon accounting. This set of rules forces companies to report their carbon footprint, i.e.,  their direct and indirect emissions. The stricter the regulation, the more precise the information that companies have to collect, and so the better they can understand their own carbon footprint. However, stricter regulations also lead to higher social costs, as more resources are needed to collect and process the necessary information. 
In France, for example, the Cour des Comptes reports the estimated cost to comply with the EU Directive on CSRD to be between 40 k\euro~and 320 k\euro~per company  and a further cost ranging between 67 k\euro~and 540 k\euro~for auditing (see \cite{Cour24}, Section 2.I.A.3). 

Thus, we assume that the Sender (the regulator) internalises the social costs induced by a stricter policy with a strictly convex cost function. Besides, we consider a continuum of firms in a mean-field approximation.  In this model, each firm's preferences are a trade-off between two objectives:  achieving a personal carbon footprint target and positioning themselves below the average carbon footprint of all firms by some factor $\epsilon$ of the standard deviation. This latter objective captures the ``best-in-class'' feature of the emulation involved in designing the regulation of carbon footprint disclosure policies. The larger the parameter $\epsilon$, the more  emulation is at play, as companies strive to be further below the industry average. The explicit closed-form equilibrium and the optimal information design policy are provided for this example as well. 

Our model of information provision shows that when there is no strategic interaction between firms, only information cost can justify deviation from sending full information. But, in the presence of strategic interaction in the form of a best-in-class emulation process, information can also be used by strategic firms to coordinate on a collectively higher level of carbon footprint to look individually cleaner. And, thus, even in the absence of information cost, it may be optimal for the regulator to blur information available to firms.

\paragraph{Literature review.} 
Our contribution lies at the intersection of two fields within mathematical economics. The first is the literature on dynamic Bayesian persuasion and the second is the application of filtering theory to economic sciences. 

Regarding dynamic persuasion, most of the literature focuses on the discrete-time setting like \cite{Renault13, Golosov14, Au15, Renault17, Bizzotto21, Bizzotto21b, Zhao24}.  All of these models, with the exception of \cite{Golosov14}, are in infinite horizon.  Furthermore, aside from \cite{Renault13}, who considers a state evolving according to a Markov chain, the canonical problem involves a binary state variable known to the Sender.  In these models, the Receiver has to make a decision at some point in time and when the decision is made the game stops. The Sender can choose an information rate at each time period and the key question is whether the Receiver can learn the true value of the state.  The Receiver is often assumed to be short-sighted in order to simplify the analysis, focusing on the Sender's dynamic information supply and avoiding the complexities that arise from considering the Receiver's rational anticipations.  Moreover, \cite{Bizzotto21} and \cite{Bizzotto21b} consider the case where the Receiver has an additional source of information.

Dynamic Bayesian persuasion in a continuous-time setting was first introduced in the seminal paper by \cite{Ely17}. This initial work was followed by a series of papers, including \cite{Orlov20, Ely19, Bizzotto21, Liao21, Ball23, Yao23, Escude23}. 

The authors of these studies assume a myopic Receiver for the same reason as in discrete-time models. 
In these models, the Sender controls either the intensity of a Poisson process modelling the rate of information provision (\cite{Ely17}) or the volatility of the estimated state.  For instance, in \cite{Liao21} the volatility can take two values, while in \cite{Escude23}  it is a continuous process taking values in a bounded interval. 
The primary focus of this literature is  on whether it is in the Sender's interest to reveal the true value of the state and whether the Receiver is able to discern it. Consequently, no cost function is attributed to the provision of information. 
The only exception we find is in \cite{Che23}, which introduces a cost for the information rate and the idea that providing information can be a costly business.  
In fact, the literature on Bayesian persuasion is mostly interested in the pure effect of the information provision on the behaviour of recipients and for this reason, most papers do not consider the presence of cost of information. 
However, as early as the example of the lobbyist in \cite[Section V]{Kamenica11}, the authors notice that the optimal solution of the lobbyist is to provide either full information or no information, which they find to be inconsistent with the observation that the lobbyist provides partial information in the form of costly studies.  
As a consequence, the authors themselves suggest that adding a cost to the provision of information would reconcile this discrepancy between theory and observation: a suggestion that our own paper takes up.

Regarding filtering theory and its relationship with economics and finance, it is widely employed in econometrics to estimate time-varying coefficients of models such as ARMA, regime-switching models, stochastic volatility models. 
Key textbook references that have advanced the use  of the Kalman filter for economic applications include \cite{Harvey89}, \cite{Hamilton94} and \cite{Durbin01}. 
Among the numerous papers in finance using filtering methods, we can mention the influential \cite{Schwartz97}, where these techniques are applied to model the dynamics of commodity prices; \cite{Engle83}, that focuses on the application to wage rates in California; and \cite{Harvey93} which applies filtering theory to macroeconomic data on business cycles.

\paragraph{Paper's roadmap.} 
The paper begins by introducing our motivation: the two applications of our framework for continuous-time Bayesian persuasion (Section~\ref{sec:apps}).  In Section~\ref{sec:theory}, we present the mathematical formalisation of the framework, that is a linear-quadratic (stochastic) Stackelberg game with an ergodic criterion in partial information between a (MFG of) Receiver(s) and a Sender. In this Section we also present and prove the Verification Theorem used to solve the Receiver's separated problem. In Section~\ref{sec:apps-sol}, we detail the numerical results for the two applications, and finally, Section~\ref{sec:conclusion} concludes the paper, providing research perspectives on the application of filtering theory in shaping agents' behaviour.

\paragraph{Notation.}
In what follows, vectors are column vectors.
For any $n,m \in \NN$, we denote the standard Euclidean norm on $\RR^n$ by $\
|\cdot|$. For matrices $A = (a_{i,j})_{i,j}\in \mathbb{R}^{n\times m}$, we use the following norm (and the induced distance)
$\| A\| := \sum_{i=1}^n \sum_{j=1}^m |a_{i,j}|$.
We use the notation $\mathbb R_+^{n \times m}$ for the set of all matrices $A$ such that $A^\top A$ is positive semi-definite. Moreover, $\mathcal S_n (\mathbb R)$ indicates the set of all symmetric real-valued $n$-dimensional squared matrices.
Finally, we denote by $\mathcal H_n ^2(\mathbb F)$ the set of stochastic processes $\nu$ with values in $\mathbb R^n$, progressively measurable with respect to $\mathbb F := {(\mathcal F_t)}_{t \ge 0}$ and such that $\mathbb E [ \int_0^T {|\nu_t|}^2 \D t ] < +\infty$ for all $T > 0$.

\section{Applications} \label{sec:apps}

\subsection{The informative value of smart meters}\label{sec:ex_elec}

The necessity to reduce the carbon emissions from electricity production worldwide has driven the large-scale development of renewable energy sources, particularly wind and solar power. However, the intermittent nature of these energy sources -- characterised by significant fluctuations in production throughout the day -- necessitates the use of storage systems. 
These systems are essential to manage sudden increases or decreases in supply and ensure a stable and reliable energy grid.  Making electricity consumption more flexible, i.e., encouraging consumers to adjust their consumption patterns to the  production, is one available way to adjust production to consumption. For this reason, the European Union has developed an intensive program of smart metering deployment for private households.  A smart meter is an advanced type of utility meter that records the consumption of electricity, gas, or water in a more detailed and precise manner than traditional meters. These meters are equipped with digital technology and communication capabilities that enable a two-way communication between the meter and the utility provider.  Across all state members, a total amount of approximately 38 billions euros has been invested for the installation of 223 millions smart meters (see \cite{Tounquet20}, Executive Summary p.~19).  The objectives of this massive deployment of smart meters are multiple: better estimation of the imbalances costs between retail supply and consumption, improvement of transmission and distribution network reliability, enhancement of dynamic tarification and provision of information signals to consumers when consumption is highly costly. 

Nevertheless, the vast collection of communication functionalities of smart meters has raised concerns about both their acceptability among consumers (see \cite{Pepermans14}) and their efficiency within specific electricity tariff structures (see \cite{Dato20, Nouicer23}). Indeed, smart meters offer social utility mainly when pricing schemes are based on time-of-use, allowing consumers to adjust their consumption in response to varying rates throughout the day.
Nevertheless, existing dynamic tarification has already succeeded in reducing electricity consumption without relying on sophisticated digital technologies (see \cite{Faruqui10}).  Indeed, an important feature of the smart meter is the frequency at which electricity consumption is measured. From daily measurements to high-frequency monitoring (meaning several kHz), the data collected on the household consumption varies significantly across this spectrum: in the first case, one can only assess daily electricity consumption trends, while at the other end of the frequency spectrum, it is possible to identify specific household appliances in use and offer tailored energy efficiency advice (see \cite{Schirmer23} for a review of non-intrusive load management applications of smart meters). Nevertheless, measuring the electricity consumption at a higher frequency increases the cost of the smart meters information system. This raises the question: what is the socially optimal trade-off?

\medskip

In what follows, we focus on the social economic value of smart meters in their role as information provision tools to consumers. We frame the problem in a persuasion framework in continuous-time. At initial time, a regulated electricity producer acting as a monopolist on the market considers investing in a {\em device}, the smart meter, designed to provide information to the consumers on their real-time consumption so that they can adjust their behaviour in response to fluctuating electricity prices. The device has several features that can be resumed by its ability to provide an accurate description of the continuous consumption of the consumer. This accuracy can be thought as the frequency at which the device provides information on the consumption of the consumers. Indeed, prior to the development of smart meters, the electricity consumption of households were measured once a year, and bills that were sent very month or every two months, were based on past estimation of the consumption. Now, with smart meters, it is possible to measure the electricity consumption at a daily basis, or an hourly basis and even at a 10-minutes basis. The consumers could reveice a bill every hour.  Would the cost of such devices be negligible, the attention of the consumer infinite, and the electricity consistent with marginal costs of production, the solution is simply to install the most precise device and to provide full information to the consumer (see \cite{Aid23}). However, it is quite clear that measuring and delivering to the consumers its consumption at a one-minute basis should be (much) more costly than a simple annual measure. Thus,  when the device is costly and consumers face frictions in changing their consumption habits, the producer has to make a trade-off choice between the rate at which information is provided to consumers and the induced investment cost. 

\medskip

We assume that consumers have the capacity to process any level of information rate but present frictions in changing their habit of consumption. 
We analyse and discuss two cases: the case where consumers are modeled as a representative agent and the case of a crowd of interacting consumers, leading to a persuasion problem with mean-field interaction of the Receivers.

\medskip

The consumption rate (in kW) of the representative consumer is assumed to follow the dynamics
\begin{align}
	\D X_t = -\kappa(X_t - \ell) \D t  + v_t \D t + \D W_t, \quad X_0=x_0>0,
\end{align}
where $\ell>0$ is the stationary level of consumption rate (his consumption habit), $\kappa > 0$ is the speed of mean-reversion to the just-mentioned stationary level, $v$ is the control of the consumer on its consumption rate and $W$ is a standard Brownian motion. The consumer has net utility (utility minus friction costs of action) resulting from his energy consumption, given by 
\begin{equation}\label{eq:utility_elec}
u(x,v) := -u_0 (x- \ell)^2 - c(v),
\end{equation}
where $u_0 >0$ and $c(v) := \frac12 v^2/\gamma$ is the cost of control and $\gamma>0$. The utility function of the consumer translates its preference for its average level of consumption $\ell$. Besides, the larger the parameter $u_0$, the larger the consumer gives importance to the induced comfort from this level of consumption. The real-time price of electricity is assumed to be an affine function of the consumption rate given by $p(x) := p_0 + p_1 x$, with $p_0, p_1>0$.  This price can be seen as the spot price of electricity. The consumer wishes to maximise the utility of electricity consumption minus friction costs and electricity bills. Namely, he solves the following stochastic control problem:
\begin{align}
	\sup_{v \in \mathcal A^{\rm R}} J^{\rm {R}}(v; b ,\sigma) := \sup_{v \in \mathcal A^{\rm R}}   \liminf_{T\to +\infty} \EE\Bigg[ \frac1T \int_0^T \Big( u(X_t,v_t) - p(X_t) X_t \Big) \D t \Bigg],
\end{align}
where the admissible set of controls $\mathcal A^{\rm R}$ is specified in Section~\ref{sec:theory}. Note that the model assumes that the instantaneous price $p(X_t)$ to be non-observable, i.e. the consumer does not observe the spot price of electricity. This translates the idea described above that the consumption and its cost are known to the consumer at a frequency which depends on the metering device. Indeed, at time $t\geq 0$, the consumer has access to a measure $M_t$ of his consumption $X_t$ thanks to the information device installed by the producer. The measure of the consumption follows the dynamics 
\begin{align}
	\D M_t = bX_t \D t  + \sigma \D B_t, \quad M_0 =0,
\end{align}
where $b$ and $\sigma$ are non-negative parameters characterising the smart meter. The measure process $M$ is the information process generated by the smart meter. When $b=0$, the consumers is left with the only information of its average consumption. This situation can be interpreted as the pre-smart metering area when consumption was measured once a year. As $b$ increases and $\sigma$ remains constant, the consumer can infer more an more rapidely its true consumption compared to what he thinks it is.  The ratio ${b}/{\sigma}$ is called hereafter {\em the precision} of the device (signal over noise ratio). Moreover, because only the ratio $b/\sigma$ matters in the design of the smart meter, we will normalise it to $1$, and reduce our smart meter design to $b$. We denote by $h(b)$ the investment cost for a device with characteristics $b$. Such a function satisfies $h(0) = 0$,  and $\lim_{b\to \infty} h(b)  = +\infty$, and $h$ strictly convex. This means that it is costless to provide no information, and it is infinitely costly to provide the continuous-time information on the process $X$ to the consumer. 

On the production side, electricity is produced at a cost $g(x)$ of the form $g(x) := g_0 x +  g_1 x^2$, with $g_0,g_1>0$.  The problem of the producer that designs the device and sells electricity is given by
\begin{align}
	\inf_{b \in \mathcal A^{\rm S}} J^{\rm {S}}(b; v) :=  \inf_{b \in \mathcal A^{\rm S}} h(b) + \limsup_{T \to \infty}  \EE\Bigg[ \frac{1}{T} \int_0^T \Big( g(X_t) - p(X_t) X_t \Big) \D t\Bigg].
\end{align}

The results from Section \ref{sec:Sender_erg} will show that the producer's problem reduces to
  \begin{align*}
     \inf_{b \in \mathcal A^{\rm S}} \widetilde J^{\rm S} (b) := & \inf_{b \in \mathcal A^{\rm S}} h(b)+ \EE  \left[ g(X^*_{\infty}) - p(X^*_{\infty}) X^*_{\infty} \right],
    \end{align*}
where $X^*_{\infty}$ denotes the stationary value of the consumption process under optimal response of the Receiver to the message process $M$.

\medskip

The case of an interacting crowd of consumers follows directly from the formulation of the $N$-consumers situation. 
We assume that all the consumers are indistinguishable and that each of them has a smart meter device with the same characteristics in his household. 
Interaction between consumers occurs through the price function which in this case is defined as $p(\bar x)$, with $N \bar x := \sum_{i=1}^N x_i$ the average consumption rate. 
The objective functional of consumer $i$ is thus
\begin{align}
&	J^{\rm {R},i}(v^i)  := \liminf_{T\to +\infty} \EE\Bigg[ \frac1T \int_0^T \Big( u(X^i_t,v^i _t) - p(\bar X_t) X^i_t \Big) \D t \Bigg], \\
 \text{where} \quad	
 & \D X^i_t = -\kappa (X^i_t - \ell) \D t + v^i_t \D t + \D W^i_t, \quad X_0^i =x_0^i>0,
\end{align}
with all $W^i$'s being independent standard Brownian motions, and $v^i$ being the control of the $i^{\text{th}}$ player. 
Thanks to the device installed in his house, each consumer has access to a message process
\begin{align}
	\D M^i_t = bX^i_t \D t  + \sigma \D B^i_t, \quad M_0^i =0,
\end{align}
where the parameters of the device are the same for all consumers and where all $B^i$'s are independent standard Brownian motions. The objective of the producer is to minimise
\begin{align}
	J^{\rm{S}}(b) :=  h(b) + \limsup_{T \to \infty} \frac{1}{T} \EE\Big[ \int_0^T \Big( \frac1N \sum_{i=1}^N g(X^i_t) - p(\bar X_t) \bar X_t \Big) \D t\Big].
\end{align}

Formally, taking the limit for $N\to\infty$ in the game above, we obtain the mean field game of the electricity demand-response matching problem, which might be described as follows.
The representative consumer is maximising his utility of electricity consumption minus adjustment cost and electricity bills, which in this case depends on the behaviour of the whole system:
\begin{align}
	&J^{\rm{R}}(v)  := \liminf_{T\to +\infty} \EE\Bigg[ \frac1T \int_0^T \Big( u(X_t,v_t) - p(m) X_t \Big) \D t \Bigg], \\
\text{where} \quad &	\D X_t = -\kappa (X_t - \ell) \D t + v_t \D t + \D W_t, 
\end{align}
for $W$ standard Brownian motion, and $v$ the control of the representative player. 
The rate of information received by the representative consumer via the device installed in his home has the following dynamics:
\begin{align}
	\D M_t = bX_t \D t  + \sigma \D B_t, \quad M_0 =0,
\end{align}
where  $B$ is an independent standard Brownian motion. 
The objective of the authority is to minimise
\begin{align}
	J^{\rm{S}}(b) :=  h(b) + \limsup_{T \to \infty} \frac{1}{T} \EE\Big[ \int_0^T \Big( g(X_t) - p(m) m \Big) \D t\Big].
\end{align}
Let us notice that the fixed point condition ensures that, at the equilibrium, it should hold
\begin{align}
	m = \lim_{T \to \infty} \EE[X_T],
\end{align}
provided such a limit exists.

\medskip

The solution and the numerical illustrations of these two problems are given in Section~\ref{ssec:res-smart}.

\subsection{Carbon footprint accounting rules}\label{sec:carbon}

The global shift towards decarbonising electricity generation through renewable energy is just one of the strategies being employed to achieve carbon neutrality by 2050, as outlined in the Paris Agreement.
Another important pillar of carbon emission reduction is the extra-financial reporting of corporate firms, specifically through ESG (Environmental, Social, and Governance) indicators.

In the European Union, extra-financial reporting has taken the form of a specific Directive, the Corporate Sustainable Reporting Disclosure (CSRD)(see  \cite{Directive22}) only in 2022. The main purpose of this Directive is to compel firms to disclose information on their level of direct and indirect emissions in order either to foster a competitive emulation among companies striving to be best-in-class (the soft version) or to expose the true sustainability levels of each firm, enabling regulators to adjust legislation and investors to modify their portfolios (the hard version). In any case, this new regulation induces additional costs to firms who must gather new data.  For instance, in France, the {\em Cour des Comptes} estimates that the cost for each firm to comply  with the EU Directive on CSRD ranges between 40 k\euro~and 320~k\euro~per year, with additional auditing costs ranging from 67~k\euro~to 540~k\euro~(see \cite{Cour24}, Sec. 2.I.A.3). 

\medskip

The CSRD, as well as alternative standards for reporting extra-financial indicators, can be interpreted as an information device to the public. The rules can be flexible, allowing firms significant freedom in how they report their emissions, which may result in limited information production and lower reporting costs.  On the contrary, as it seems to be the trend, the reporting rules can be strict, specific and inflexible, leaving firms with little choice but to bear higher reporting costs, ultimately leading to the production of higher-quality information. 

\medskip

We design a persuasion model to assess the trade-off faced by a regulator between enforcing a high level of carbon footprint reporting stringency and the induced costs incurred by firms. To remain in the general framework presented in Section~\ref{sec:theory}, we assume that the Receivers of the information are the firms themselves. Their objective is to achieve a compromise between reducing their carbon footprint to a target level and to be below the average carbon footprint of their peers, in a form of best-in-class emulation process.

\medskip

We consider the following dynamics for the carbon footprint of the representative agent of the population of firms
\begin{align}
    \D X_t = -\kappa (X_t  - \ell) \D t +  v_t \D t + \D W_t, \quad X_0 = x_0,
\end{align}
where $\ell>0$ represents the long term average of carbon footprint, $\kappa$ is the speed of mean-reversion and the process $(v_t)_{t\geq 0}$ is a possible action that the firm can make to change its footprint. 
The representative firm minimises a criterion given by
\begin{align}
    J^{\rm R}(v) 
    := \limsup_{T\to +\infty} \frac{1}{T}\int_0 ^T \EE\Big[ c(v_t) + \lambda_q (X_t - q)^2 + \lambda_a (X_t-a)^2 \Big] \D t,
\end{align}
with $\gamma>0$, $c(v) := \frac12 \frac{v^2}{\gamma}$ the cost of control and $\lambda_a, \lambda_q$ non-negative parameters and where $q \in \RR$ is to be determined at equilibrium by the (mean-field) consistency condition:  
\begin{align*}
q =\lim_{T\to +\infty}\mathbb E[X_T] - \epsilon \sigma_T,
\end{align*} 
where $\sigma_T$ is the standard deviation of $X_T$, $\epsilon$ is the {\em level of best-in-class target} and $\sigma_\infty$ is the stationary standard-deviation of the carbon footprint process. 
We note that $\lambda_q$ does not depend on $q$. It is just a notation we have set to recall the fact that this is the weight on the best-in-class target $q$ in the objective function of the firm.

When $\epsilon=0$, the representative firm only wishes to be close to the average of the crowd of firm. 
On the other hand, when $\epsilon>0$, it means that the representative firm wishes to have a carbon footprint below the average, i.e.,  to be a best-in-class company in terms of carbon footprint reduction. 
The non-negative parameters $\lambda_q$ and $\lambda_a$ capture the firm's preferences between striving to be the best-in-class  versus sticking to his own target carbon footprint, $a$, that can actually differ from the carbon footprint natural value of $\ell$. For a $\lambda_q$ close to $1$ and $\lambda_a$ near to zero, firms have a high tendency to lower their carbon footprint below the average level of their peers. This features captures  preferences in the spirit of ranking games (see \cite{Bayraktar16} and the references therein).

The accounting rules on carbon footprint is represented as a device $(b,\sigma) \in \mathbb R \times (0,+\infty)$ such that the information process $M$, whose dynamics is given by
\begin{align}
    \D M_t = bX_t \D t  + \sigma \D B_t, \quad M_0 =0,
\end{align}
provides an information to the firm on its true carbon footprint. 
In particular, when $b=0$, the accounting norms are so mild that firms have no information on their carbon footprint. 

The regulator's objective is to minimise the potential damage caused by the firms' carbon footprint. To capture this, we consider a simple quadratic cost function, yielding to the following optimisation criterion:
\begin{align}
    \inf_{(b,\sigma) \in \mathcal A^{\rm S}} h(b,\sigma) +  \limsup_{T \to \infty} \frac{1}{T} \int_0 ^T \frac12 \EE\big[   c (X_t^*)^2\big] \D t, 
\end{align}
with $c>0$ and where $X^\ast$ is the carbon footprint process at the mean-field equilibrium. 
The regulator wishes to find the most cost-effective set of reporting rules to achieve the maximum reduction of firms' carbon footprint. 

\medskip

The solution and numerical illustration of this problem are provided in Section~\ref{ssec:res-carbon}.

\section{Linear-quadratic ergodic persuasion game}\label{sec:theory}

Let us consider a complete probability space $(\Omega, \mathcal{F}, \mathbb{P})$ endowed with two independent Brownian motions, $W$ and $B$, of dimensions $d_W$ and $d_B$. We denote by $\mathbb F^{W,B} = (\mathcal F_t^{B,W})_{t\geq 0}$ the filtration generated by $B,W$ and completed with the $\mathbb P$-negligible events in $\mathcal F$, so fulfilling the usual conditions of right-continuity and $\mathbb P$-completeness.
 
We focus on a game between a \emph{Sender} and a \emph{Receiver} of information, also called the agents, who are indicated, respectively, by S and R.  
Borrowing the terminology from stochastic filtering theory, the $\RR^{d_W}$-valued stochastic process $X= (X_t)_{t \geq 0}$ is called unobservable \textit{signal}, since the Receiver has no access to it, while he has access to the $\RR^{d_B}$-valued process $M = (M_t)_{t \geq 0}$. 
The latter represents the \emph{message} that the Sender is providing to him. We assume the following linear dynamics for the state variables: 
\begin{align}	
	& \D X_t = (A_{\rm x} X_t + B_{\rm x} v_t + c_{\rm x})\D t +  \D W_t , \quad X_0=x_0, \label{dynX-gen} \\
	& \D M_t = b X_t \D t +  \Sigma \D B_t, \quad M_0=0  \label{dynM-gen}
\end{align}
where $A_{\rm x} \in \mathcal S_{d_W}(\RR) $, $B_{\rm x}\in \RR^{d_W\times r}$ ($r$ is the dimension of the control $v$), $c_{\rm x} \in \RR^{d_W}$, $\Sigma \in \RR^{d_B \times d_B}$, $b \in \RR^{d_B \times d_W}$, and $x_0 \in\RR^{d_W}$ is a real constant.
The Receiver's control is any stochastic process $v$ with values in $\RR^r$ satisfying suitable measurability and integrability conditions that we give later in this section. The Sender's control is a deterministic, time-independent pair of matrices $(b, \Sigma)\in \RR^{d_B\times d_W} \times \RR^{d_B \times d_B}$.

The Receiver's information is modelled by the complete and right-continuous filtration $\mathbb F^{\rm R} := (\mathcal F_t ^{\rm R})_{t \geq 0}$ defined as
\begin{equation}\label{eq:filtrations_R}
    \mathcal{F}_t^{\rm R} := \bigcap_{\epsilon >0} \sigma(M_s: 0 \le s \le t+ \epsilon ) \vee\mathcal N, \quad t \geq 0,
\end{equation} 
where $\mathcal N$ is the set of all $\mathbb P$-null events in $\mathcal F$. In other words, the Receiver has access only to the information incorporated in the message, that is in turn controlled by the Sender.

The Receiver and the Sender aim at minimising the so-called long-time average cost functionals, defined, respectively, by
\begin{align*}
    J^{\rm R} (v;b,\Sigma)& :=\limsup_{T \to \infty}\frac{1}{T}\EE \left[ \int_0^T f(X_t,v_t) \D t \right],
    \\
    J^{\rm S} (b,\Sigma ; v) &:= \limsup_{T\to \infty} \frac{1}{T} \EE \left[\int_0^T g(X_{t},v_t) \D t\right]+ h(b,\Sigma),  \end{align*}
where
\begin{equation}\label{eq:f}
    f(x,v)  = x^\top F_2 x +F_1 ^\top x + F_0 + v^\top C_2 v + C_1 ^\top v , \quad (x,v) \in \RR^{d_W + r},
\end{equation}
with 
$F_2 \in \mathcal S_{d_W}(\RR)$, $F_1 \in \RR^{d_W}$, $F_0 \in \mathbb R$, $C_2 \in \mathcal S_r(\RR)$, and  $C_1 \in \RR^{r}$.  
Furthermore, we set $h : \RR^{d_B \times d_W} \times \RR^{d_B \times d_B} \to \RR$, measurable, to be the cost associated to the controls of the Sender and $g$ to be a suitable measurable function that will be specified in the applications.

We need some standing technical assumptions on the coefficients. Before, let us give some definitions that we use to formulate them.
\begin{definition}\label{def:stab_obs}
Given matrices $A \in {\mathbb R}^{n \times n}, B \in {\mathbb R}^{n \times m}$ and $C \in {\mathbb R}^{p \times n}$, we say that 
\begin{itemize}
\item[a)] The pair $(A, B)$ is {\em{stabilisable}} if there exists a constant matrix $K \in {\mathbb R}^{m \times n}$ such that $A + B  K$ is {\em{stable}} (namely, all its eigenvalues have negative real parts).
\item[b)] The pair $(C, A)$ is {\em{observable}} (or equivalently $(A^\top, C^\top)$ is {\em{controllable}}) if the $(n \times np)$-matrix $\Gamma(A^\top,C^\top) =[ C^\top, A^\top C^\top , \ldots, {A^\top}^{n-1} C^\top ] $ has rank $n$.
\end{itemize}
\end{definition}
Now, we can give our standing assumptions.

\begin{assumption}\label{ass_coeff} Let the coefficients satisfy the following properties:
	\begin{itemize}
		\item[i)] The pair $(A_{\rm x}, b^\top)$ is {\em{stabilisable}};
		\item[ii)] $F_2 \ge 0$ (positive semi-definite);
		\item[iii)] $ B_{\rm x} C_2^{-1} B_{\rm x}^\top \ge 0$ (positive semi-definite);
		\item[iv)] $(-A_{\rm x}, B_{\rm x} C_2^{-1} B_{\rm x}^\top)$ stabilisable;
		\item[v)] $(F_2,-A_{\rm x})$ detectable, or equivalently $(-A_{\rm x}^\top,F_2^\top)$ stabilisable.
	\end{itemize}
\end{assumption}

As mentioned above, the Sender's admissible control set is given by a fixed subset of all deterministic controls $(b,\Sigma) \in \RR^{d_B \times d_W} \times \RR^{d_B \times d_B}$ with $\Sigma$ invertible. Hence we denote by $\mathcal A^{\rm S}$ the set of all such controls,
while the set of all admissible control processes for the Receiver is given by the following definition.

\begin{definition}\label{adm-case1} 
    We say that any $\RR^r$-valued stochastic process $v$ is \emph{admissible} if it belongs to $\mathcal H_r ^2(\mathbb F^{B,W})$, it is $\mathbb F^{\rm R}$-progressively measurable and it satisfies
    \[ \lim_{T \to \infty} \frac{1}{T} \mathbb E[ |X_T |^2] =0. \]
   The set of admissible processes is denoted by $\mathcal A^{\rm R}$.
   \end{definition}

\begin{remark}
Observe that the filtration $\mathbb F^{\rm R}$ depends implicitly on the control process $v$, that we omitted from the notation for simplicity. Indeed, $v$ affects $X$, which in turns affects $M$, which generates $\mathbb F^{\rm R}$. So, defining the Receiver's admissible controls as simply those processes which are $\mathbb F^{\rm R}$-progressively measurable would have resulted in some circularity. The definition above addresses this issue as follows: first, given any $v \in \mathcal H_r ^2(\mathbb F^{B,W})$, the state variable $(X,M)$ as well as the filtration $\mathbb F^{\rm R}$ are well defined and they all depend on $v$; second, we have restricted Receiver's choice to those controls in $\mathcal H_r ^2(\mathbb F^{B,W})$ which are also $\mathbb F^{\rm R}$-progressive measurable and satisfy the limiting condition above. This procedure is pretty standard in stochastic control with partial observation, see, e.g., the recent paper \cite{cohen25} and the discussion therein. \end{remark}

Let us notice that the linear-quadratic structure of the Receiver's objective functional together with the linear dynamics of the state grant the existence and uniqueness of solutions for the system in Equation \eqref{dynM-gen} over $[0,T]$, for each $T>0$, and the well-posedness of  the Receiver's control problem.
Furthermore, notice that they are consistent with the general admissibility conditions stated in Section \ref{sec:ver_thm}.
\medskip

The  problem is modelled via a Stackelberg stochastic differential game with ergodic criteria where the leader and the follower are, respectively, the Sender and the Receiver. Since the leader moves first, she needs to solve the Receiver's problem in order to find her optimal control. Therefore the game is solved according to the following two steps:
\begin{itemize}   
\item[1.] 
    For a fixed Sender's control $(b,\Sigma)\in \mathcal A^{\rm S}$, we solve first the corresponding Receiver control problem
    \[ \inf_{v \in \mathcal A^{\rm R}} J^{\rm R} (v;b,\Sigma). \]
    Let us denote by $v^*(b,\Sigma)$ its unique solution, which possibly depends on the Sender's control $(b,\Sigma)$. We compute the solution in the next sub-section. \medskip
    
\item[2.] 
    After step 1 above, the Sender knows the optimal control that the Receiver applies for any of her controls $(b,\Sigma)$. So, she optimises her objective functional by solving 
    \begin{align}\label{eq:JS}
        \inf_{(b,\Sigma) \in \mathcal A^{\rm S}} J^{\rm S} (b,\Sigma; v^* (b,\Sigma)) := & \inf_{(b , \Sigma)\in \mathcal A^{\rm S}} \left\{ \limsup_{T\to \infty} \frac{1}{T} \EE \left[\int_0^T g(X^*_{t},v^*_t (b,\Sigma)) \D t\right] + h(b, \Sigma)  \right\} 
        \qquad \qquad 
    \end{align}
    where $X^*$ is the state controlled by the Receiver's optimal strategy $v^* (b,\Sigma)$. 
\end{itemize}

 Thus, we first provide the optimal control of the Receiver for a given Sender's admissible control $(b,\Sigma)$ (in Section~\ref{ssec:Receiver_gen}) and then we turn to the optimisation of the Sender's criteria (in Section~\ref{sec:Sender_erg}).

\subsection{The Receiver's problem}\label{ssec:Receiver_gen}

Let $(b,\Sigma) \in \mathcal A^{\rm S}$. We recall that the Receiver's problem is given by
\begin{align*}
    \inf_{v \in \mathcal A^{\rm R}}  J^{\rm R}(v;b,\Sigma) := \inf_{v \in \mathcal A^{\rm R}} \limsup_{T \to \infty}\frac{1}{T}\EE \left[ \int_0^T f(X_t,v_t) \D t \right]
\end{align*}
where the dynamics of the state $X$ is as in Equations \eqref{dynM-gen}. Notice that the objective functional depends on $(b,\Sigma)$ via the filtration $\mathbb F^{\rm R}$. Indeed the latter is generated by the message $M$, which is controlled by the Sender.

\begin{remark}\label{rem:sigma}
    The parameters $b$ and $\Sigma$, which are the deterministic controls of the Sender, are linked to the so-called informational gain that the Receiver can get from the observation of the message process $M$. In particular, the larger the quantity $\widetilde{b}:= \Sigma^{-1}b$, the faster for the Receiver to estimate the value of the hidden process $X$. 
    Indeed, the model introduced before is equivalent to the one where the process $M$ is replaced by
    	\begin{align}\label{eq:dyn-tildeM}
    		d\widetilde{M}_t &= \widetilde{b} X_t \D t +   \D B_t, \qquad \widetilde{M}_0=0,
    	\end{align}
    with $\widetilde{b} \in \mathbb{R}$.
    Indeed, it suffices to multiply $M$ and $b$ by the inverse of the matrix $\Sigma$, namely by introducing 
    $\widetilde{M}:=\Sigma^{-1} M$ and  
    $\widetilde{b}:=\Sigma^{-1} b$, with $\Sigma^{-1} \in \RR^{d_B \times d_B}$.
    In particular, the limit case $\| \Sigma \| \to 0$, which represents the full information case, corresponds to $|\widetilde b | \to +\infty$. Observe that the filtration of the Receiver does not change after such a normalisation. Thus, for the rest of the paper we assume without loss of generality that $\Sigma$ is the identity matrix. Indeed, if not we can always reduce to that case by means of the transformation above. Notice that, assuming $\Sigma = \mathbf 1$, the set of admissible Sender's controls, that is still denoted $\mathcal A^{\rm S}$, is a fixed subset of all matrices $b \in \RR^{d_B \times d_W}$. In particular, in the applications, where $d_W = d_B =1$, we have $\mathcal A^{\rm S} = \RR_+ := [0,+\infty)$. 
\end{remark}

\subsubsection{Filtering and problem reformulation}
According to the remark above, let us assume from now on that $\Sigma$ is the identity matrix. Let us recall that the Receiver has to estimate the unobserved process $X$ given the observation process $M$, with dynamics:
\begin{align}	
	& \D X_t = (A_{\rm x} X_t + B_{\rm x} v_t + c_{\rm x})\D t +  \D W_t , \quad X_0=x_0,  \\
	& \D M_t = b X_t \D t +  \D B_t, \quad M_0=0. 
\end{align}
This is a well known task in the literature, known as stochastic filtering. In the uncontrolled case, i.e., $B_{\rm x} \equiv 0$, Kalman and Bucy were the first to solve the problem and for this we refer, e.g., to \cite{Kalman63} or to \cite[Section 4.4]{Davis_book}. In the case when $B_{\rm x} \neq 0$ this problem has also already been studied and we refer to \cite[Section 5.3]{Davis_book}.

In mathematical terms, we are interested in the following projection, the filter process of $X$:
\begin{align}
    \widehat{X}_t := \EE[X_t | \mathcal{F}_t^{\rm R}], \quad t \ge 0,
\end{align}
where $\mathcal F_t^{\rm R}$ was defined in Equation\eqref{eq:filtrations_R}. 
Following \cite[Equation (5.49)]{Davis_book} we find
\begin{align} \label{dynobs-gen}
    d\widehat X_t  = (A_{\rm x} \widehat X_t + B_{\rm x} v_t + c_{\rm x}) \D t +  P_{b} (t) b^\top dI_t, \quad \widehat X_0 = x_0,
\end{align}
where the stochastic process $I$, called the innovation process, is the $\mathbb F^{\rm R}$-Brownian motion
$$
dI_t  = \D M_t - b \widehat X_t \D t, \quad t \ge 0,
$$
and where $P_b$ is the $(d_W \times d_W)$-dimensional error covariance matrix
\begin{equation}\label{eq:P_b}
P_b (t) = \EE \left[ (X_t -\widehat{X}_t) (X_t -\widehat{X}_t)^\top \big| \Ff^{\rm R}_t \right] = \widehat{X^2 _t}- (\widehat{X}_t)^2 = \EE \left[ (X_t -\widehat{X}_t)(X_t -\widehat{X}_t)^\top \right]  
\end{equation}
which is a deterministic function of time, independent of the control, and solving the matrix Riccati equation (see \cite{Davis_book}, Section 5.3 and Equation (4.40)):
\begin{equation}\label{eq:ODE_P}
 P'_b (t) = A_{\rm x} P_{b}(t) + P_{b}(t) A_{\rm x}^\top - P_{b}(t) b^\top  b P_{b}(t) + \mathbf{1}, \quad P_{b}(0) = \V[X_0]=\V[x_0] = 0.
\end{equation}

Since we are interested in solving an ergodic control problem, we need to check that the system has a good ``long-run'' performance, i.e., we have to study the stability of the filter at infinity. As well explained in \cite[Section 5.4]{Davis_book}, in the case when the coefficient matrices do not depend on time, this boils down to studying the asymptotic properties of the matrix Riccati equation \eqref{eq:ODE_P}.  

Exploiting \cite[Theorem 2.1]{Wohnam68}, we have that, under Assumption \ref{ass_coeff} $i)$, there exists a positive definite matrix $P_b(\infty)$, which reads:
\begin{equation}
P_b(\infty) = \lim_{t \rightarrow +\infty} P_b(t).
\end{equation}
In this case, $P_b(\infty)$ is the unique positive semi-definite solution of the Algebraic Riccati Equation (ARE) 
\begin{equation}
A_{\rm x} P_b + P_b A_{\rm x}^\top - P_b b^\top b P_b +  \mathbf{1} =0
\end{equation}
and the $(d_W \times d_W)$-matrix $A_{\rm x} - b^\top b P_\infty$ is stable.

\begin{remark}
Theorem 2.1 in \cite{Wohnam68} also requires the pair $(\mathbf 1, A_{\rm x})$ to be observable, which is automatically satisfied.
\end{remark}

Using the tower property of conditional expectation given $\mathcal F_t ^{\rm R}$, the $\mathcal F^{\rm R}$-progressive measurability of the control $v$ and the linear-quadratic structure of $f$, we can rewrite the Receiver's objective functional as
\begin{align*}
    J^{\rm R}(v;b) 
    & = \limsup_{T \to \infty}\frac{1}{T}\EE \left[ \int_0^T f(X_t,v_t) \D t \right]\\
    & = \limsup_{T \to \infty}\frac{1}{T}\Bigg\{ \EE \left[ \int_0^T f(\widehat{X}_t,v_t) \D t \right] + \int_0^T \underbrace{\EE \left[ (X_t -\widehat{X}_t)^\top F_2 (X_t -\widehat{X}_t) \right]}_{\sum_{i,j=1}^{d_W} F_2 ^{i,j} P^{i,j} _b(t) }\D t\Bigg\}.
\end{align*}
Thus, given the fact that $P_b(t)$ in Equation \eqref{eq:P_b} is a deterministic function of time independent of the Receiver's control, neglecting the constant terms, we are led to consider the following equivalent problem
\begin{align} \label{Receiver_pb}
    \inf_{v \in \mathcal A^{\rm R}}  \widetilde J^{\rm R}(v ; b) 
    & := \inf_{v \in \mathcal A^{\rm R}} \limsup_{T \to \infty}\frac{1}{T}\EE \left[ \int_0^T f (\widehat{X}_t,v_t) \D t \right],
\end{align}
where the dynamics of the observable state $\widehat{X}$ is given in \eqref{dynobs-gen}. 
Now, we notice that, due to the presence of the time-dependent function $P_b$, we are dealing with a stochastic control problem with an ergodic criterion but time-inhomogeneous state dynamics. To the best of our knowledge, this case is not covered by the ergodic stochastic control literature. We are going to solve it by means of the following verification theorem, that we state below for a slightly more general setting. 

\subsubsection{A time-inhomogeneous version of the verification theorem with ergodic criterion}\label{sec:ver_thm}
Let $(\Omega, \mathcal F, \mathbb P)$ be a complete probability space equipped with a filtration $\overline{\mathbb F}:=(\overline{\mathcal F}_t)_{t \geq 0}$, satisfying the usual conditions of right-continuity and $\mathbb P$-completeness, supporting an $m$-dimensional $\overline{\mathbb F}$-Brownian motion $\overline W$. We consider the stochastic control problem
\[ 
\inf_{\nu \in \mathcal U} J(\nu) := \inf_{\nu \in \mathcal U} \limsup_{T \to \infty} \frac{1}{T} \mathbb E\left[ \int_0 ^T f(Y_t , \nu_t) \D t \right],
\]
where $Y = Y^\nu$ is an $n$-dimensional stochastic process solving
\[ 
\D Y_t = \mu(Y_t, \nu_t) \D t + \sigma(t) \D \overline W_t, \quad Y_0 = y_0 \in \mathbb R^n,
\]
with $\mu : \mathbb R^n \times \mathbb R^d \to \mathbb R^n$ measurable, $\sigma(t)$ semi-positive definite and deterministic matrix in $ L^2(\mathbb R_+^{n \times m})$,
and $\nu$ belonging to the set of admissible controls $\mathcal U$, which is the set of all processes in $\mathcal H^2 _d (\overline{\mathbb F})$ such that
\[\lim_{T \to +\infty} \frac{1}{T}\mathbb E\left[ | Y^\nu _T|^2 \right] = 0.\]
Moreover we assume that $\mu (y,\nu)$ is Lipschitz continuous in $y$ uniformly in $\nu,$ so that $Y$ is well defined as the unique strong solution of the SDE above. Finally, the usual a-priori $L^2$ estimates on $Y$ holds true, i.e., for all $q \geq 1$ and $T >0$, there exists some constant $C_{T,q}>0$ such that \[\mathbb E\left[\sup_{t \in [0,T]}| Y^\nu _t |^q \right] \leq C_{T,q}(1+|y_0|^q).\]
Notice that the dynamics of the state variable $Y$ is time-inhomogeneous due to coefficient $\sigma(t)$. We can nonetheless provide a verification theorem under the assumption that $\sigma(t)$ has a limit for $t \to \infty$, which allows to proceed as in the standard time-homogeneous case.
\begin{theorem}[Verification theorem] Assume that there exists $\sigma_\infty \in \mathbb R_+^{n \times m}$ such that
\begin{equation}\label{lim-sigma} \frac{1}{T}\int_0 ^T \| \sigma^\top (t)\sigma (t) - \sigma^\top _\infty\sigma_\infty \| \D t \to 0, \quad T \to +\infty.\end{equation}
Moreover, let $V \in C^2 (\mathbb R^n)$ with at most quadratic growth and let the Hessian matrix $D^2 V$ be bounded, and $\zeta \in \mathbb R$ such that
\begin{equation}\label{HJB-erg}
\inf_{u \in \mathbb R^d} \big (\mathcal L_\infty ^u V(y) + f(y,u)\big) - \zeta = 0, \quad y \in \mathbb R^n,
\end{equation}
where
\[ \mathcal L_\infty ^{u} \varphi (y):= \mu(y,u) \nabla \varphi (y) + \frac12 \text{Tr} \left( \sigma^\top _\infty D^2 \varphi(y) \sigma_\infty \right) ,\]
for all functions $\varphi \in C^2(\mathbb R^n)$.
Assume that there exists a measurable function $\nu^*: \mathbb R^n \to \mathbb R^d$ attaining the infimum in \eqref{HJB-erg}. 
Then, if $\nu^*_t := \nu^*(Y_t)$, $t\geq 0$, belongs to $\mathcal U$, it is an optimal control and $\zeta = J(\nu^*)$.
\end{theorem}

\begin{proof} Let $\nu \in \mathcal U$ be given. Applying It\^o's formula to $V(Y_t)$ over $[0,T]$, dividing by $T$ and taking the expectation on both sides lead to
\begin{align*}
\frac{1}{T} \mathbb E [ V(Y_T) - V(y_0)] &= \frac{1}{T} \mathbb E\left[\int_0 ^T \mathcal L_t ^{\nu_t} V(Y_t) \D t +\int_0 ^T \nabla V(Y_t) \sigma(t) d \overline W_t \right] \\
& = \frac{1}{T} \mathbb E\left[\int_0 ^T \mathcal L_t ^{\nu_t} V(Y_t) \D t \right],
\end{align*}
where the above It\^o integral has zero expectation thanks to the fact that its integrand is square-integrable, i.e., $\mathbb E[\int_0 ^T |\nabla V(Y_t) \sigma(t)|^2 \D t]< \infty$, for all $T>0$. This follows from the $\mathcal C^2$-regularity of $V$, the boundedness of the Hessian matrix $D^2 V$, the square integrability of $\sigma(t)$ and the a-priori estimates of $Y$.
Adding and subtracting $\mathcal L_\infty ^{\nu_t} V(Y_t)$ and $f(Y_t, \nu_t)$ inside the integral in the RHS and using the PDE \eqref{HJB-erg} yield
\[ \frac{1}{T} \mathbb E [ V(Y_T) - V(y_0)] \geq \frac{1}{T} \mathbb E\left[\int_0 ^T (\mathcal L_t ^{\nu_t} - \mathcal L_\infty ^{\nu_t})V(Y_t) \D t \right] - \frac{1}{T} \mathbb E\left[\int_0 ^T f(Y_t, \nu_t) \D t\right] + \zeta.\]
Using assumption in Equation \eqref{lim-sigma} and the fact that $D^2 V$ is bounded, we obtain
\begin{align*}
\frac{1}{T} \mathbb E\left[\int_0 ^T |(\mathcal L_t ^{\nu_t} - \mathcal L_\infty ^{\nu_t})V(Y_t)| \D t \right] & = \frac{1}{2T} \mathbb E\left[\int_0 ^T |\textrm{Tr} \left( \sigma^\top(t) D^2 V(Y_t) \sigma(t) \right) - \textrm{Tr} \left( \sigma^\top _\infty D^2 V(Y_t) \sigma_\infty \right)| \D t \right] \\
& \le  \frac{1}{2T} \mathbb E\left[\int_0 ^T \| \sigma^\top (t)\sigma (t) - \sigma^\top _\infty \sigma_\infty \| \cdot \| D^2 V(Y_t)\| \D t \right] 
\end{align*}
which goes to zero for $T \to +\infty$.
Moreover, since $\nu$ is admissible and $V$ has at most quadratic growth, we have 
\[ \frac{1}{T} \mathbb E [ V(Y_T) - V(y_0)] \to 0, \quad T \to +\infty.\]
Therefore we obtain
\[ \zeta \leq \limsup_{T \to \infty} \frac{1}{T} \mathbb E\left[\int_0 ^T f(Y_t, \nu_t) \D t\right],\]
and, since $\nu$ is arbitrary in $\mathcal U$, we have $\zeta \leq \inf_{\nu \in \mathcal U} J(\nu)$. If we repeat the same arguments for the control $\nu^*_t = u^*(Y_t)$, where the function $u^*(y)$ maximises the PDE \eqref{HJB-erg}, we get equalities everywhere leading to $\zeta = J(\nu^*)$. Since $\nu ^*$ belongs to $\mathcal U$ by assumption, we can conclude that it is an optimal control.
\end{proof}

\subsubsection{Solving the Receiver's problem} In order to apply the verification theorem in our setting, to the Receiver's problem, we notice that  
$n=d_B+d_W$, $m=d_B$, $d=r$, with $\nu=v$ and $Y=\widehat X$, $\overline W = I$, and the drift and volatility coefficients read
\begin{align}
    \mu(x, v)= A_{\rm x} x + B_{\rm x} v + c_{\rm x},
    \quad \text{ and } \quad
    \sigma(t) = P_b(t)b^\top.
\end{align}
We need to check that $\sigma(t)$ verifies assumption in Equation \eqref{lim-sigma}. 
Let us recall from the previous section that $P_b(t)$ is the solution of the Riccati ODE \eqref{eq:ODE_P}, which admits a finite limit $P_b(\infty)$ for $t \to +\infty$, as established before.

Hence, it remains to show that $\frac{1}{T}\int_0 ^T |P_b(t) - P_b(\infty)|^2 \D t \to 0$, as $T \to \infty$. This follows from the elementary fact that for any given bounded measurable function $h:\mathbb R_+ \to \mathbb R_+$, such that $\lim_{t \to \infty} h(t) = 0$, one has $T^{-1}\int_0 ^T h(t) \D t \to 0$, for $T \to +\infty$.\medskip

Now, let us exploit the PDE \eqref{HJB-erg} to solve the Receiver's problem. 
In this case, Equation \eqref{HJB-erg} specialises to
\begin{align}
    \inf_{v \in \mathbb R^r} & \left\{  \nabla  V(x)^\top B_{\rm x} v + v^\top C_2 v +C_1^\top v \right\} 
    + \nabla  V(x)^\top (A_{\rm x} x +c_{\rm x}) \\
 &   + x^\top F_2 x + F_1 ^\top x + \frac{1}{2}\text{Tr}(b P_b^\top D^2 V(x)P_b b^\top)  - \zeta  =0,
\end{align}
where for the sake of a compact notation we have written $P_b$ in place of $P_b(\infty)$.
Thus, the first order condition leads to
\begin{equation}
     v^*(x) = - \frac{1}{2}C_2 ^{-1} \left( B_{\rm x} ^\top \nabla V(x) + C_1\right),
\end{equation}
leading in turn to
\begin{align}
    &\nabla  V^\top (A_{\rm x} x  - \frac{1}{2}B_{\rm x} C_2 ^{-1}(B_{\rm x} ^\top \nabla V + C_1) +c_{\rm x}) + x^\top F_2 x + F_1 ^\top x 
    + \frac{1}{4} (\nabla V^\top B_{\rm x} + C_1^\top) C_2 ^{-1}(B_{\rm x} ^\top \nabla V + C_1)  \nonumber \\
    & - \frac{1}{2}C_1 ^\top C_2 ^{-1}(B_{\rm x} ^\top \nabla V + C_1) + \frac{1}{2}\text{Tr}(b P_b^\top D^2 V P_b b^\top) = \zeta,
\end{align}
where we have written $\nabla V$ in place of $\nabla V(x)$ for short. We propose the following natural ansatz for the solution:
\begin{align}
    V(x)= x^\top G_2 x + G_1 ^\top x,
\end{align}
where $G_1 \in \RR^d$ and $G_2 \in \mathcal S_{d_W}(\RR)$ are parameters to be determined.  
By plugging the ansatz into the ODE above we get the following: 
\begin{align}
    &(2G_2 x +G_1)^\top (A_{\rm x} x  - \frac{1}{2}B_{\rm x} C_2 ^{-1}(B_{\rm x} ^\top (2G_2 x + G_1) + C_1) + c_{\rm x}) + x^\top F_2 x + F_1 ^\top x \\
    &+\frac{1}{4}(B_{\rm x} ^\top (2G_2 x + G_1) + C_1)^\top C_2 ^{-1} (B_{\rm x} ^\top (2G_2 x + G_1) + C_1) \\
    & - \frac{1}{2} C_1 ^\top C_2 ^{-1} (B_{\rm x} ^\top (2 G_2 x + G_1) + C_1) + \text{Tr}(bP_b^\top G_2 P_b b^\top) = \zeta.
\end{align}
By identifying quadratic and linear terms (with respect to $x$) and by using the symmetry of $A_{\rm x}$, we obtain the following system of equations for the couple $(G_1,G_2)$:
\begin{equation}
\label{eq:sys_G_i}
\left\{
\begin{array}{rcl}
A_{\rm x}^\top G_2 + G_2 A_{\rm x} - G_2 B_{\rm x} C_2 ^{-1} B_{\rm x} ^\top G_2 + F_2 &=& 0 \vspace{0.2cm} \\
(-G_2 B_{\rm x} C_2^{-1} B_{\rm x} ^\top + A_{\rm x} ^\top )G_1 + 2G_2 c_{\rm x} - G_2 B_{\rm x} C_2 ^{-1} C_1 + F_1  &=& 0
\end{array}
\right.
\end{equation}
while identifying the constant leads to the following expression of $\zeta$ in terms of $G_1,G_2$ and other coefficients of the model:
\[
\zeta = c_{\rm x} ^\top G_1 - \frac{1}{4} G_1^\top B_{\rm x} C_2^{-1} (B_{\rm x} ^\top G_1 + C_1) - \frac{1}{4} C_1 ^\top C_2 ^{-1} (B_{\rm x} ^\top G_1 + C_1)  + \text{Tr}(b P_b^\top G_2 P_b b^\top).
\]
Observe that the first equation in the system \eqref{eq:sys_G_i} is an ARE that can be solved independently of the second one. 
Existence and uniqueness (in a sense to be made precise) of a symmetric solution $G_2$ to the ARE 
\begin{align*}
 G_2 B_{\rm x} C_2 ^{-1} B_{\rm x} ^\top G_2 - A_{\rm x}^\top G_2-  G_2 A_{\rm x}- F_2 = 0,
\end{align*}
can be proved under Assumptions \ref{ass_coeff} $ii)-v)$ on $F_2 \in \mathcal S_{d_W}(\RR)$ and $B_{\rm x} C_2^{-1} B_{\rm x}^\top \ge 0 \in \mathcal S_{d_W}(\RR)$ (recall Definition \ref{def:stab_obs}).
Indeed, it suffices to use \cite[Corollary 8.1.11]{Lancaster95}, with $A=-A_{\rm x}$, $D= B_{\rm x} C_2^{-1} B_{\rm x}^\top$ , $C=F_2$, to obtain existence of a real symmetric solution $G_2$ to the ARE above and uniqueness of a real symmetric solution $G_2^*$ such that $-A_{\rm x} + B_{\rm x} C_2^{-1} B_{\rm x}^\top G_2^* $ is stable.

For this unique $G_2^*$, that for notational simplicity we denote by $G_2$, the vector $G_1$ is well defined and it is given by 
\begin{equation}\label{eq:G1}
     G_1 = \left(- G_2 B_{\rm x} C_2 ^{-1} B_{\rm x} ^\top + A^\top_{\rm x} \right)^{-1}(G _2 B_{\rm x} C_2 ^{-1} C_1 - 2 G_2 c_{\rm x} - F_1),
\end{equation}
provided the matrix $\Theta_1 := A_{\rm x} - B_{\rm x} C_2^{-1} B_{\rm x} ^\top G_2 \in \mathcal S_{d_W}(\RR)$ is invertible, which is granted by the following assumption (see also Section \ref{ssec:stationary} for another use of this assumption):
\begin{assumption}\label{ass_Theta1}
	All the eigenvalues of the matrix $\Theta_1$ have strictly negative real part (i.e., it is Hurwitz stable).
\end{assumption}

Thus, the optimal control for the Receiver's problem is given by
\begin{align}\label{eq:v*}
    v^*_t  = - \frac{1}{2}C_2 ^{-1}\left( B_{\rm x} ^\top(2 G_2 \widehat X_t ^* + G_1) + C_1 \right),
\end{align}
where
\begin{align}	
	& \D \widehat X_t^* = (A_{\rm x} \widehat X_t^* + B_{\rm x} v_t^* + c_{\rm x})\D t + P_b(t) b^\top \D I_t , \quad \widehat X_0 =x_0. 
\end{align}
It remains to prove that the controlled state variable above satisfies the admissibility condition $\lim_{T \to +\infty} \frac{1}{T}\mathbb E[ | \widehat X^{*} _T|^2 ] = 0$. This is a consequence of the existence of a Gaussian stationary distribution for the process $(X^*, \widehat X^*)$, that is addressed in the following section. Indeed if $(X^*, \widehat X^*)$ has a Gaussian stationary distribution, the same holds for $\widehat X^*$. Therefore $\mathbb E [ | \widehat X^{*} _T|^2 ]$ has a finite limit for $T \to \infty$, hence the admissibility condition above is fulfilled.
This in turn implies that $v^*$ is an admissible control.

\subsection{The stationary distribution}\label{ssec:stationary}
In this part we compute $\Xx^*:=( X^*, \widehat{X}^*)^\top$'s stationary law, that  is  $\lim_{T\to \infty} \mathcal L(\Xx^*_T)$. Possibly extending the initial probability space,  we define  a random variable $\Xx^*_\infty$, such that  $\lim_{T\to \infty} \mathcal L(\Xx^*_T)=\mathcal L(\Xx^*_\infty)$.

By plugging the expression of the optimal control $v^*$ in the dynamics of $\mathcal X^*$ we obtain
\begin{align}
    \D \Xx^*_t
    = \Bigg[ 
    \begin{pmatrix}
        A_{\rm x} & - B_{\rm x} C_2 ^{-1} B_{\rm x} ^\top G_2\\
        \bold{0} & A_{\rm x} -  B_{\rm x} C_2 ^{-1} B_{\rm x} ^\top G_2
    \end{pmatrix}
    \Xx^*_t +
    \begin{pmatrix}
        -\frac{1}{2} B_{\rm x} C_2 ^{-1} (B_{\rm x} ^\top G_1 + C_1) + c_{\rm x}\\
        -\frac{1}{2} B_{\rm x} C_2 ^{-1} (B_{\rm x} ^\top G_1 + C_1) + c_{\rm x} 
    \end{pmatrix}
    \Bigg] \D t +
    \begin{pmatrix}
        \D W _t\\
        P_b(t)b^\top \D I_t 
    \end{pmatrix}
\end{align}
where we recall that $\D I_t = \D M_t - b\widehat{X}^* _t \D t = b(X^*_t - \widehat{X}_t^*)\D t + \D B_t$, with $B$ and $W$ independent Brownian motions. Therefore, we rewrite the states' dynamics as 
\begin{equation}
    \D \Xx^*_t
    = \left( \Theta(t)
    \Xx^*_t +
   \vartheta
    \right) \D t +
\Xi(t)  (
        \D W_t ,
       \D B_t )^\top , \quad \mathcal X^* _0 =(x_0,x_0) ^\top,
\end{equation}
where
\begin{align*}  \Theta(t) & :=
\begin{pmatrix}
        A_{\rm x} & - B_{\rm x} C_2^{-1} B_{\rm x} ^\top G_2\\
        P_b(t) b^\top b & A_{\rm x} - B_{\rm x} C_2^{-1} B_{\rm x} ^\top G_2 - 
            P_b(t)b^\top b
    \end{pmatrix},\\
     \vartheta & := \begin{pmatrix}
     \theta\\
     \theta
     \end{pmatrix}
     :=
    \begin{pmatrix}
        -\frac{1}{2} B_{\rm x} C_2^{-1} (B_{\rm x} ^\top G_1 +C_1) +c_{\rm x}\\
        -\frac{1}{2} B_{\rm x} C_2^{-1} (B_{\rm x} ^\top G_1 +C_1) +c_{\rm x}
    \end{pmatrix},
    \end{align*}
    and
    \[ \Xi(t) : =
    \begin{pmatrix}
        \mathbf 1 & \mathbf 0\\
        \mathbf 0 & P_b(t)b^\top 
    \end{pmatrix},
    \]
 which is a square matrix with dimension $d_W + d_B$.
This is a system of linear SDEs with deterministic coefficients, hence the stationary distribution (when it exists) is Gaussian. To characterise it, it suffices to compute mean and covariance matrix for any fixed $t$ and pass to the limit. 
Let us now set the following notations (recall that $\EE[X^*_t] = \EE[\widehat X^*_t], t \ge 0$)
\begin{align}\label{Eq:notations_ODEs}
     m_t & :=   \EE[X^*_t] \in \RR^{d_W}, \\
     w_t & := \textrm{Cov}(\Xx^*_t) = \EE[(\Xx^*_t- \EE(\mathcal X_t^*))(\Xx^*_t-\EE(\mathcal X_t^*))^\top] \in \mathcal S_{2 d_W}(\RR).
\end{align}
By \cite[Chapter 5, Problem 6.1]{karatzas14}, we have that $m_t$ and $w_t$ solve the following ODEs
\begin{align}\label{eq:ODE_m}
    \D m_t & = 
    \left(
   \Theta_1  m_t + \theta
    \right) \D t , \quad m_0 = x_0^\top \\ \label{eq:ODE_w}
    \D w_t & = ( w_t \Theta(t)^\top + \Theta(t) w_t +\Xi(t)\Xi(t)^\top)  \D t ,  \quad w_0 = \mathbf 0,
\end{align}
where the symmetric matrix $\Theta_1 := A_{\rm x} - B_{\rm x} C_2^{-1} B_{\rm x} ^\top G_2 \in \mathcal S_{d_W}(\RR)$ already appeared in the solution $G_1$ in Equation \eqref{eq:G1}.
We need to check under which assumptions on the coefficients the solutions $m_t$ and $w_t$ the linear ODEs above have limits $m_\infty,w_\infty$ for $t \to +\infty$.

\paragraph{The ODE for the mean.}
Under Assumption \ref{ass_Theta1}, the ODE \eqref{eq:ODE_m} can be transformed by defining $\bar m_t= m_t + \Theta_1^{-1} \theta$, which solves: $d\bar m_t= \Theta_1 \bar m_t \D t$.
Hence, $m_t \to m_\infty$, as $t \to +\infty$, if and only if $\bar m_t \to \mathbf 0$, as $t \to +\infty$, where $m_\infty = \Theta_1^{-1} \theta$. It is well known that this happens (namely the equilibrium point $\bar m = \mathbf 0$ is globally asymptotically stable) if and only if all the eigenvalues of $\Theta_1$ have strictly negative real parts (see, e.g., \cite[Theorem 3.5]{Khalil96}).

\paragraph{The ODE for the covariance.}
The ODE \eqref{eq:ODE_w} is a special case of a matrix Riccati equation (such as Equation \eqref{eq:ODE_P}), without the quadratic term. In order to apply Theorem 2.1 in \cite{Wohnam68}  we set the following: 
\begin{assumption}\label{ass_stab_cov}
	$\Theta$ is stable and $(\Xi, \Theta^\top )$ is observable. 
\end{assumption}
An application of \cite[Theorem 2.1]{Wohnam68} ensures that there exists a positive definite matrix $w_\infty = \lim_{t \rightarrow +\infty} w_t$, which is the unique positive semi-definite solution of the ARE: 
\begin{equation}\label{eq:Lyapunov}
w_\infty \Theta_\infty ^\top + \Theta_\infty w_\infty +\Xi_\infty \Xi_\infty^\top =\bold{0}.
\end{equation}
\begin{remark}
	The above Equation \eqref{eq:Lyapunov} is a (continuous time) Lyapunov equation (see, e.g., \cite[Section 3.3., page 123]{Khalil96}) and, when $\Theta_\infty$ is Hurwitz stable, it can be solved explicitly as
	\[  w^* _\infty = \int_0 ^\infty e^{\Theta_\infty t} \Xi_\infty \Xi_\infty ^\top e^{\Theta_\infty ^\top t} \D t. \]
\end{remark}

\subsection{The Receivers' ergodic mean-field game}\label{ssec:MFG_erg}

Let us now consider a mean-field game (MFG) generalisation of the initial one-Receiver problem, as introduced at the beginning of Section \ref{sec:theory}, to address a case where we have a large population of identical Receivers getting information from the same Sender. MFGs have been pioneered by \cite{lasrylions07} and \cite{huang06} as ``natural'' limit for large population stochastic differential games where the players are identical and they are linked by a mean-field interaction. In case of an ergodic criterion for all players, the latter means that each player interacts with the long term empirical distribution of the population's states and (possibly) actions. Therefore, when passing to the limit with respect to the number of players, it suffices to study the problem of a representative player interacting with a distribution over the state space, which at equilibrium is equal to the stationary distribution of the representative player's state. When the objective functional is linear-quadratic as it is the case here, the distribution over the state space appears along the resolution process only through its stationary mean and (possibly) covariance. We refer to the two-volume book \cite{carmona-delarue-book} for a full probabilistic treatment of MFGs and for a wide range of applications. Finally, among many papers on MFG with ergodic criterion, we refer to \cite{bardi13} and \cite{bardi14} for the analytic approach in the linear-quadratic case and to \cite{cohen23}, \cite{cohen24} and \cite{ferrari23} for the probabilistic approach.\smallskip

We shortly describe the MFG analogue of the Receiver problem we solved above. The state variables have the following dynamics:
\begin{align}\label{eq:SDEs-MFG}
\D X_t &= (A_{\rm x} X_t + B_{\rm x} v_t + c_{\rm x}(m,w))\D t +  \D W_t , \qquad X_0= x_0,\\ \nonumber
\D M_t &= b X_t \D t +  \D B_t, \qquad M_0=0,
\end{align}
where $A_{\rm x}, B_{\rm x}$ are as in the one-receiver setting and $c_{\rm x}:\mathbb R^{d_W} \times \mathcal S_{d_W} \to \RR^{d_W}$ measurable.
Let us notice that $m$ and $w$ are set to be equal to the mean and the covariance, respectively, of the stationary distribution of the state process $X$ at equilibrium.

The representative Receiver's filtration is still given by $\mathbb  F^{\rm R}$ defined as before.
Furthermore, for a fixed $(m,w)$, the Receiver's objective functional is given by
\begin{align*}
 J^{\rm R}(v; b, m,w) 
 := & \limsup_{T \to \infty}\frac{1}{T}\EE \left[ \int_0^T f(X_t,v_t, m ,w) \D t \right],
\end{align*}
with
$$
f(x,v, m) = x^\top F_2 (m,w) x +F_1 (m,w) ^\top x +v^\top C_2 v + C_1 ^\top v , \quad (x,v) \in \RR^{d_W + r},
$$
where $C_1, C_2$ are as in the one-receiver setting, whereas $F_1: \RR^{d_W} \times \mathcal S_{d_W} \to \RR^{d_W}$ and $F_2: \RR^{d_W} \times \mathcal S_{d_W}  \to \RR^{d_W \times d_W}$ are given measurable functions. We will see specific examples for the functions $c_{\rm x}, F_1,F_2$ in the applications. All assumptions of the one-Receiver case are in force (for all $(m,w) \in \RR^{d_W} \times \mathcal S_{d_W} $), so that all results from the previous sections can be used in this part as well.

The representative Receiver wishes to minimise his objective functional over the set $\mathcal A^{\rm R}$ of all admissible controls defined as in the one-receiver framework.

In this setting, a MFG Nash equilibrium is any triple $(v^*,m^*, w^*) \in \mathcal A^{\rm R} \times \mathbb R^{d_W} \times \mathcal S_{d_W} $ such that:
\begin{itemize} 
\item[(i)] (optimality) $v^*$ is optimal given $(m^*,w^*)$;
\item[(ii)] (fixed point) $m^* = \mathbb E[X_\infty ^*]$ and $w^* = \text{Cov}(X_\infty ^*)$, where $X_\infty ^*$ is a random variable following the stationary distribution of the state $X^*$ (provided such distribution exists). 
\end{itemize} 
We also introduce the notation $m^*_\infty := \EE[X^*_\infty]$ and $w^*_\infty := \text{Cov}[X_\infty ^*]$.
For the optimality condition (i), let the parameter $(m,w) \in \RR^{d_W} \times \mathcal S_{d_W} $ be fixed. The results of the previous section give immediately the optimal strategy of the representative Receiver $v^* (m,w)$ and the stationary mean of the state process $X^*$ as functions of $(m,w)$:
\begin{align}
	v_t ^* (m,w) & =  - \frac{1}{2}C_2 ^{-1}(B_{\rm x} ^\top(2 G_2 (m,w) \widehat X_t ^* + G_1 (m,w)) + C_1) \\ 
	m_\infty ^{*\top}(m,w)  & = -\Theta_1(m,w) ^{-1} \theta(m,w) \nonumber\\
	& = -(A_{\rm x} - B_{\rm x} C_2^{-1} B_{\rm x} ^\top G_2(m,w)) ^{-1}\left(-\frac{1}{2} B_{\rm x} C_2^{-1} (B_{\rm x} ^\top G_1(m,w) +C_1) +c_{\rm x}(m,w)\right) ,
\end{align}
whereas the stationary covariance $w_\infty ^*(m,w)$ is given as the unique solution of the Lyapunov equation
\begin{align}
	w_\infty \Theta_\infty (m,w) ^\top + \Theta_\infty (m,w) w_\infty +\Xi_\infty (m,w)  \Xi_\infty^\top (m,w) =\bold{0},
\end{align}
where $G_1(m,w),G_2(m,w), \Theta_1(m,w), \Theta_\infty (m,w), \Xi_\infty (m,w)$ and $\theta(m,w)$ are defined as in previous section, except that here we have the extra dependence of these coefficients on $m$ and $w$ due to the presence in the systems of $F_1(m,w), F_2(m,w)$ and $c_{\rm x}(m,w)$, which depend on $(m,w)$ by assumption.

Finally, for the fixed point condition (ii) to hold at the equilibrium we impose 
\[ m^* _\infty (m,w) = m, \quad w^*_\infty (m,w) = w,\]
leading, in particular, to the extra equation
\begin{equation}\label{eq:fixed-pt}
	m^\top =  -(A_{\rm x} - B_{\rm x} C_2^{-1} B_{\rm x} ^\top G_2(m,w)) ^{-1}\left(-\frac{1}{2} B_{\rm x} C_2^{-1} (B_{\rm x} ^\top G_1(m,w) +C_1) +c_{\rm x}(m,w)\right).
\end{equation}
To conclude, assume that there exists a solution $(G_1^*, G_2^*, m^*, w^*)$ to the system
\begin{equation}
\label{eq:sys_G_i_m}
\left\{
\begin{array}{rcl}
A_{\rm x}^\top G_2 + G_2 A_{\rm x} - G_2 B_{\rm x} C_2 ^{-1} B_{\rm x} ^\top G_2 + F_2 (m,w) & = & 0 \vspace{0.2cm} \\
(\frac{1}{2}G_2 B_{\rm x} C_2^{-1} B_{\rm x} ^\top + A_{\rm x} ^\top )G_1 + 2G_2 c_{\rm x} - G_2 B_{\rm x} C_2 ^{-1} C_1 + F_1 (m,w)  &= & 0 \vspace{0.2cm}\\
m^\top + (A_{\rm x} - B_{\rm x} C_2^{-1} B_{\rm x} ^\top G_2 ) ^{-1}\left(-\frac{1}{2} B_{\rm x} C_2^{-1} (B_{\rm x} ^\top G_1 +C_1) +c_{\rm x}(m,w)\right) &= & 0 \vspace{0.2cm} \\
w \Theta_\infty ^\top + \Theta_\infty  w +\Xi_\infty \Xi_\infty^\top  & =& \bold{0}
\end{array}
\right.
\end{equation}
then the pair $(v^*,m^*)$ where $v_t ^* := v_t ^*(m^*, w^*)$, for all $t\geq 0$, is a MFG Nash equilibrium. The existence of a solution for the system above is not trivial due to the dependence on $(m,w)$ in some of the coefficients, namely $F_1,F_2,c_{\rm x}$. So, instead of studying it in full generality by adding even more technical assumptions, we solve it only in the two applications we are interested in (see Section \ref{sec:apps-sol}).  

\subsection{The Sender's problem}\label{sec:Sender_erg}
After solving the one-receiver and the Receivers' MFG problems, we turn to the Sender optimisation problem, where, without loss of generality,  we assume that $\Sigma$ is the identity matrix (see Remark \ref{rem:sigma}) and denote by $h(b)$ the cost $h(b,\mathbf 1)$ with a slight abuse of notation. In this case, we recall that $\mathcal A^{\rm S}$ is a fixed subset of $\mathbb R^{d_W \times d_B}$ (e.g., in the applications $\mathcal A^{\rm S} = \RR_+$).

In this section, we explain the resolution method without going into much detail. Such a method is used to completely solve the two applications. Detailed computations are postponed to Section \ref{sec:apps-sol}.

Consider first the one-receiver case. The Receiver's optimal strategy $v^*_t$ is given by \eqref{eq:v*}, so that it can be expressed as a linear function of $\widehat X^* _t$, denoted $v^*(x)$. Observe that such a control $v^*$ depends on the Sender's control $b$ only indirectly through the (projected) state variable $\widehat X^*$, whose volatility coefficients is $P_b(t)b^\top$. 

Thanks to the existence of a stationary distribution for the pair of state variables $(X^*,\widehat X^*)$ established in Section \ref{ssec:stationary}, the mean ergodic theorem (see, e.g., \cite[Theorem 1.5.18]{araposthatis11}) implies that, when the Sender's cost function $g$ satisfies the integrability condition:
\begin{equation}\label{ass:g}
g(X_\infty ^*, v^* (\widehat X^* _\infty)) \in L^1 (\mathbb P),
\end{equation}
one has
\[  
	\lim_{T \to \infty} \frac{1}{T} \mathbb E\left[ \int_0 ^T g(X_t ^*, v^*(\widehat X_t ^*))\D t \right] = \mathbb E\left[g(X_\infty ^*, v^*(\widehat X_\infty ^*))\right].
\]
Thus, the Sender's optimisation problem \eqref{eq:JS} is actually equivalent to the following:
    \begin{align}\label{eq:JS2}
      \inf_{b \in \mathcal A^{\rm S}} \widetilde J^{\rm S} (b; v^*) := & \inf_{b \in \mathcal A^{\rm S}} \EE \left[g(X^*_{\infty},v^*(\widehat X^* _\infty))\right]+ h(b),
    \end{align}
    where the random variables $(X^*_{\infty},\widehat X^*_\infty)$ are defined on a suitable probability space such that their joint law
    \begin{align}
        \mathcal{L}(X^*_{\infty},\widehat{X}^*_\infty)
        =\lim_{t \to \infty}\mathcal{L}(X^*_t, \widehat{X}^*_t),
    \end{align}
    in the topology of weak convergence of measures. Therefore, the Sender's problem reduces to a static optimisation problem, whose objective depends only on the stationary distribution of state variables. 
    
    Clearly, the same equivalence result for the Sender also holds in the MFG case, except that now $v^*$, hence the Sender's value function $\tilde J(b;v^*)$, depends on $(m^*,w^*)$ as well. 
    
    Typical examples of the function $g$ satisfying the condition in Equation \eqref{ass:g} above are considered in the applications, where the Sender's problem is also solved and discussed in full detail.

\section{Applications results}\label{sec:apps-sol}
\subsection{The informative value of smart meters}\label{ssec:res-smart}

The parameter identification (example-general theory) and the verification of the standing assumptions in this case is provided in Appendix \ref{app_ex_1}. We start with the first example, introduced in Section \ref{sec:ex_elec}. 
So, direct application of the general case developed in Section \ref{ssec:Receiver_gen} provides the Receiver's optimal control (see Equation \eqref{eq:v*})
\begin{align}\label{opt_contr_energy}
&v^* _t = \kappa \big(\widehat X_t^* - \ell) - \beta \widehat X^\ast_t + \beta m_\infty, 
\quad \beta:= \sqrt{\kappa^2+ 2\gamma (p_1+u_0)} > \kappa, \\ 
& \text{where}  \quad m_\infty :=  \frac{\kappa^2 + 2\gamma u_0 }{\kappa^2+2 \gamma(u_0+p_1)}\Big[\ell - \frac{\gamma p_0}{\kappa^2+2\gamma u_0} \Big] < \ell,
\end{align}
and the dynamics of the state process $X^*$ and its posterior estimation $\widehat X^*$ by the Receiver are
\begin{align*}
&  \D X^\ast_t = -\kappa \Big[ X^\ast_t - \widehat X^\ast_t  + \frac\beta\kappa\big(\widehat X^\ast_t - m_\infty\big)\Big] \D t + \D W_t, \quad
 d\widehat X^\ast_t = -\beta\big[\widehat X^\ast_t - m_\infty \big] \D t + b P_{b} (t) dI_t,
\end{align*}
where we recall that the dynamics of the innovation process $dI_t$ is given by $dI_t = \D M_t - b\widehat X^* _t \D t =  b(X^*_t - \widehat X_t^*)\D t + \D B_t$,
with $B$ and $W$ independent Brownian motions. Recall that,  the volatility $\sigma$ of the information process $M$ has been  normalised to $1$. 
\medskip

The ``posterior estimation'' of the consumption, namely $\widehat X^\ast$, is an Ornstein-Uhlenbeck process with intensity of mean-reversion $\beta$ (larger than the uncontrolled speed of mean-reversion, $\kappa$) and long-term value $m_\infty$ (lower than the desired habit value of the consumer, $\ell$). The dynamics of the consumption $X^\ast$ can also be interpreted as a mean-reverting process but with a stochastic long-term value which is the posterior estimation minus a correction term.
Note that both processes share the same stationary value which is precisely $m_\infty$. 
This long-term value does not depend on the information provided by the smart meter device. It only takes into account the preferences of the consumer ($u_0$ and $\ell$), the cost  ($\gamma$), the price of electricity parameters ($p_0$ and $p_1$) and the speed of mean-reversion ($\kappa$). 
When the consumer's preferences are high or the prices are low, $m_\infty$ is close to $\ell$. 
In the opposite situations, the consumer reduces his average long-run consumption habit. 
In a full information situation, the  speed of mean reversion of the consumption process $X^*$ would be $\beta$.

Despite the fact that the information provision device precision, $b$, does not change the long-run average value of the consumption, nevertheless it changes its trajectories. Figure~\ref{fig:DR}~(b), (c) and (d) illustrates three trajectories of the processes $X^\ast$ and $\widehat X^\ast$ for three different values of the precision parameter $b$ but for the same realisation of the random shock process $W$. The figure makes clear that as the precision of the device increases, the posterior estimation $\widehat X ^\ast$ gets closer to the true consumption process $X^\ast$ but, at the same time, the process $X^\ast$ is itself affected by the control of the consumer. 

In this model, average consumption reduction is not achieve thanks to information provision but thanks to prices which operate as the efficient tools to align consumers' preferences with production costs. The fact that  in our model the average consumption is left unchanged by the information device is a desirable property. Indeed, it captures the idea that even though consumers may change their consumption rate at a given instant, rebound effects maintain their consumption  on average. In this sense, the communication device is not evaluated in this capacity to reduce the consumers' consumption, but as device to enhance flexibility of consumption by moving consumption rate across time.

\medskip

The stationary distribution for the equilibrium state $X^*$ is given by the Gaussian distribution $\mathcal{N}(m_\infty, \sigma_\infty^2(b))$ with
\begin{align}
 & \sigma_\infty^2(b) = \frac{1}{2\kappa} \Big[1-\frac{\beta-\kappa}{\beta}\Big(\sqrt{1 + \frac{\kappa^2}{b^2}} -  \frac{\kappa}{b}\Big)^2\Big].
\end{align}

The smart meter only influences the variance of the stationary distribution of the consumption through the precision $b$ of the device. The variance $\sigma_\infty^2$ is a non-increasing function of the precision $b$ with limit values: 
\begin{align}\label{eq:sigmalim}
	\lim_{b\searrow0}\sigma_\infty^2(b)=\sigma_\infty^2(0^+) = \frac{1}{2\kappa} =: \bar\sigma^2, 
	\quad 
	\lim_{b\to+\infty}\sigma_\infty^2(b)=\sigma_\infty^2(+\infty) =  \frac{1}{2\beta} =: \underline\sigma^2. 
\end{align}

Furthermore, after simplification, the Sender's problem in Equation \eqref{eq:JS2} can be reduced to
\begin{align}\label{eq:JSred}
 \inf_{b\geq 0} J^{\rm S} (b) := &  \inf_{b\geq 0}   \big(g_1 -  p_1\big) \sigma_{\infty}^2(b)  + h(b),
\end{align}
where $h$ is the cost function of the device and where we restricted $b$ to be positive, as only its square intervenes in the optimisation problem of the Sender. Thus, as already exhibited in \cite{Aid23}, it is necessary that  the price of electricity is not greater than the marginal cost of production for the Sender to have an incentive to reduce the volatility. Note that, in the case where $g'(x) = 2 g_1 x = 2 p_1 x$, i.e., the electricity is priced at exactly the marginal cost of production, there is no benefit in developing an information device such as smart meters. In our model, only imperfections in the transmission to the consumers of certain production costs justify the development of an information solution. From now on, we assume that $0< g_1-p_1$ $=$ $k_{\rm I}$, where $k_{\rm I}$ stands for the {\em cost pricing imperfection}.

\medskip

Moreover, one can be surprised that the criterion of the Sender makes use of the variance of the consumption itself and not of the estimation error between the consumption $X^*$ and its posterior estimate $\widehat X^*$. Nevertheless, it holds that the variance of error of estimation, namely $\mathbb E[(\widehat X^*_\infty - X_\infty^*)^2] = P_b(\infty)$ satisfies
\begin{align*}
P_{b}(\infty)  = \frac{1}{b^2}\big(\sqrt{ \kappa^2 + b^2 } -\kappa\big),
\end{align*}
which is also a non-increasing function of the precision of the device $b$. Thus, reducing the error of estimation of the true value of the state is equivalent to reducing the variance of the state process.

\medskip

Finding a closed-form expression of the optimisation problem of the Sender~\eqref{eq:JSred} can be cumbersome depending on the form of the function $\sigma^2_\infty(b)$. Nevertheless, it is possible to reformulate the problem of the Sender using a cost function of the variance reduction, as in
\begin{align*}
 \inf_{z \in (\underline\sigma^2, \bar\sigma^2]}    \frac12 k_{\rm I} z  + H(z), 
\end{align*}
where $z := \sigma^2_\infty(b)$ and $H$ is such that $H(\bar\sigma^2) =0$ and $H(\underline\sigma^2) = +\infty$. This equivalent representation expresses the minimisation problem of the Sender directly using the cost of reducing the variance of the consumption. Thus, it is immediate that the optimal variance reduction should satisfy 
\begin{align*}
H'(z) = - \frac12 k_{\rm I},
\end{align*}
namely marginal cost equates marginal benefit which is constant here and equal to the cost pricing imperfection. Assume now the following form for the function $H$
\begin{align*}
H(z) = 
\frac{r \eta}{2}  \Big(\frac{\Delta\sigma^2}{z - \underline\sigma^2}-1\Big), \quad z \in (\underline\sigma^2, \bar\sigma^2], \quad \Delta\sigma^2 := \bar\sigma^2 - \underline\sigma^2,
\end{align*}
with $\eta>0$ a parameter of the cost function and $r$ the discount rate.  Then, it holds that
\begin{align*}
(\sigma^\ast)^2 := 
	\begin{cases}
	\bar\sigma^2 & k_{\rm I} \Delta\sigma^2 \leq  r \eta , \\
	& \\
	\underline\sigma^2 + \sqrt{\frac{r\eta \Delta\sigma^2}{k_{\rm I}}}, &   r \eta < k_{\rm I} \Delta\sigma^2.
	\end{cases}
\end{align*}
Figure~\ref{fig:DR}~(a) provides a representation of $\sigma^2_\infty(b)$ together with the optimal volatility value $\sigma^\ast = \sigma_\infty(b^\ast)$. The product $k_{\rm I} \Delta\sigma^2$ captures two relevant terms in the decision process of the Sender: the higher it is, the higher the value of an information device. This quantity can be large either because of a significant imperfection of pricing cost $k_{\rm I}$ or because the potential variance reduction $\Delta\sigma^2$ is large.
 
 \medskip

We turn now to the case of a mean-field of consumers in interaction through the electricity price. For a given value of $m$, the optimal control of the consumer is given by
\begin{align}\label{opt_contr_energy}
v^\ast _t &= \kappa \big(\widehat X_t^\ast  -\ell\big) -\beta \widehat X^\ast_t + 
\beta \hat \ell(m), \quad
\beta:= \sqrt{\kappa^2 + \gamma u_0}, \\
\text{where} \quad \widehat \ell(m) & := \frac{1}{\beta^2}\big[ (\kappa^2 + 2 \gamma u_0)  \ell   -\gamma p(m) \big],
\end{align}
and the dynamics of the consumption process and its posterior estimation are
\begin{align*}
& \D X^\ast_t = -\kappa \Big[ X^*_t -\widehat X^\ast_t + \frac{\beta}{\kappa}\big(\widehat X^\ast_t - \hat \ell(m) \big) \Big] \D t + \D W_t, \quad
 d\widehat X^\ast_t =   -\beta \big[\widehat X^*_t  -  \hat \ell(m)\big] \D t +  b P_{ b }(t) dI_t.
\end{align*}

The mean-field equilibrium is given by the fixed point condition $m = \hat \ell(m)$, yielding the equilibrium long-term value $m^\ast_\infty$ given by
\begin{align}
m^\ast_{\infty} = 
\frac{(\kappa^2  + 2\gamma u_0)  \ell - \gamma p_0}{\beta^2 + p_1\gamma} = 
\frac{(\kappa^2  + 2\gamma u_0)  \ell - \gamma p_0}{\kappa^2 + \gamma(p_1+u_0)} < m_\infty.
\end{align}
At equilibrium, the long-run consumption of a  crowd of undistinguishable consumers is lower than the one of a representative agent. But, this result is an effect of the mean-field interaction not an effect of information provision.  Furthermore, at equilibrium, the stationary distribution of the process $X^*$ is the Gaussian distribution $\mathcal{N}(m^\ast_\infty, (\sigma^\ast_\infty)^2(b))$
with
\begin{align}
\sigma^\ast_\infty(b) = \sigma_\infty(b).
\end{align}
But, in the case of a mean-field of consumers, the objective function of the Sender reduces to
\begin{align}\label{eq:JSmfgred}
 \inf_{b\geq 0} J^{\rm S} (b,\sigma) := &     g_1 \sigma_{\infty}^2(b)  + h(b),
\end{align}
because the income term $p(m^\ast_\infty)m^\ast_\infty$ does not depend on $b$. Hence, the cost pricing imperfection increases from $k_{\rm I} = g_1 - p_1$ to a strictly greater value $k_{\rm I}^{\rm mf} := g_1$. Appropriate volatility reduction follows with a higher threshold.

\medskip

Since our results are in closed-form, we may even risk ourselves in a back of the envelop real-life economic estimation of the interest in the current policy to develop smart meters for consumer's information purposes. Based on the former results in \cite[Sect. 6]{Aid22}, which performs volatility estimation of UK consumers thanks to the public data made available by the Low Carbon London project, we can take an initial volatility $\sigma = 85$~W.h$^{-\frac12}$. Using monetary incentive mechanisms, they show that this volatility can be divided by two. Thus, assume we take $\bar\sigma = 85$~W$^2$.h$^{-1}$ and $\underline\sigma = 44$~W$^2$.h$^{-1}$. For the cost pricing imperfection in the mean-field case, we can roughly take also the value picked in the same reference (sec.~A.6.6) of $g_1 = 40$~\euro.h.MW$^{-2}$ $=4\,10^{-5}$~\euro.h.W$^{-2}$. Finally to calibrate $\eta$ we assume that the current European policy that leads a spending of approximately 150~\euro~per smart meter corresponds to a division by two of the consumer's volatility. Thus we get that $\eta = 2 \times 150$~\euro. Taking a discount rate for public investment of $r =4$\%, we get $r \eta = 12$ whereas $k_{\rm I}^{\rm mf}\Delta\sigma^2 = 4\,10^{-5} \times 5290 = 0.2$ $<$ $r\eta$. We lack one or two orders of magnitude to justify the benefit of smart meters for the sole purpose of information provision to consumers. 

\begin{figure}[tbh!]
\begin{center}
\begin{tabular}{c c}
(a) & (b)   \\
\includegraphics[width=0.45\textwidth]{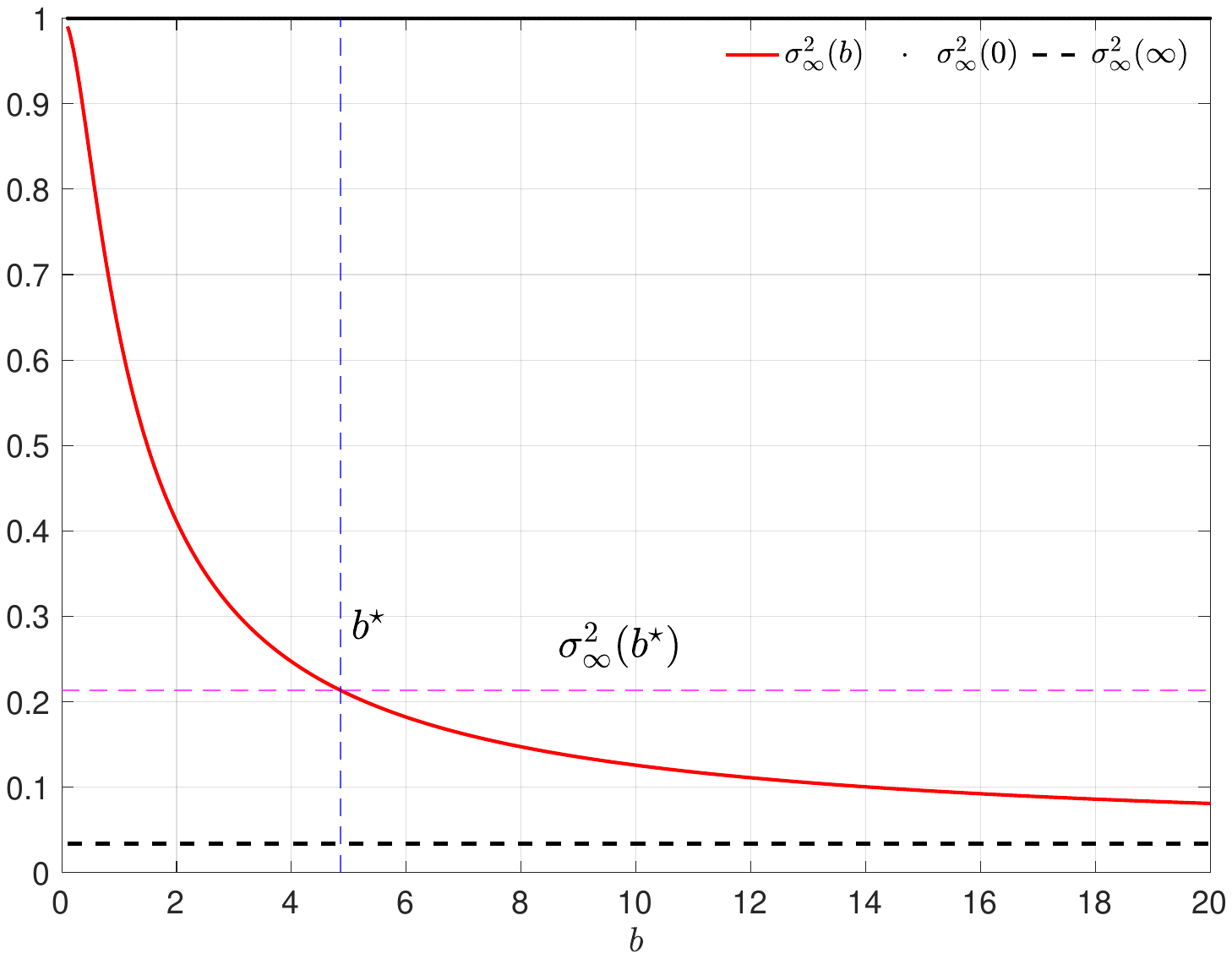} &
\includegraphics[width=0.45\textwidth]{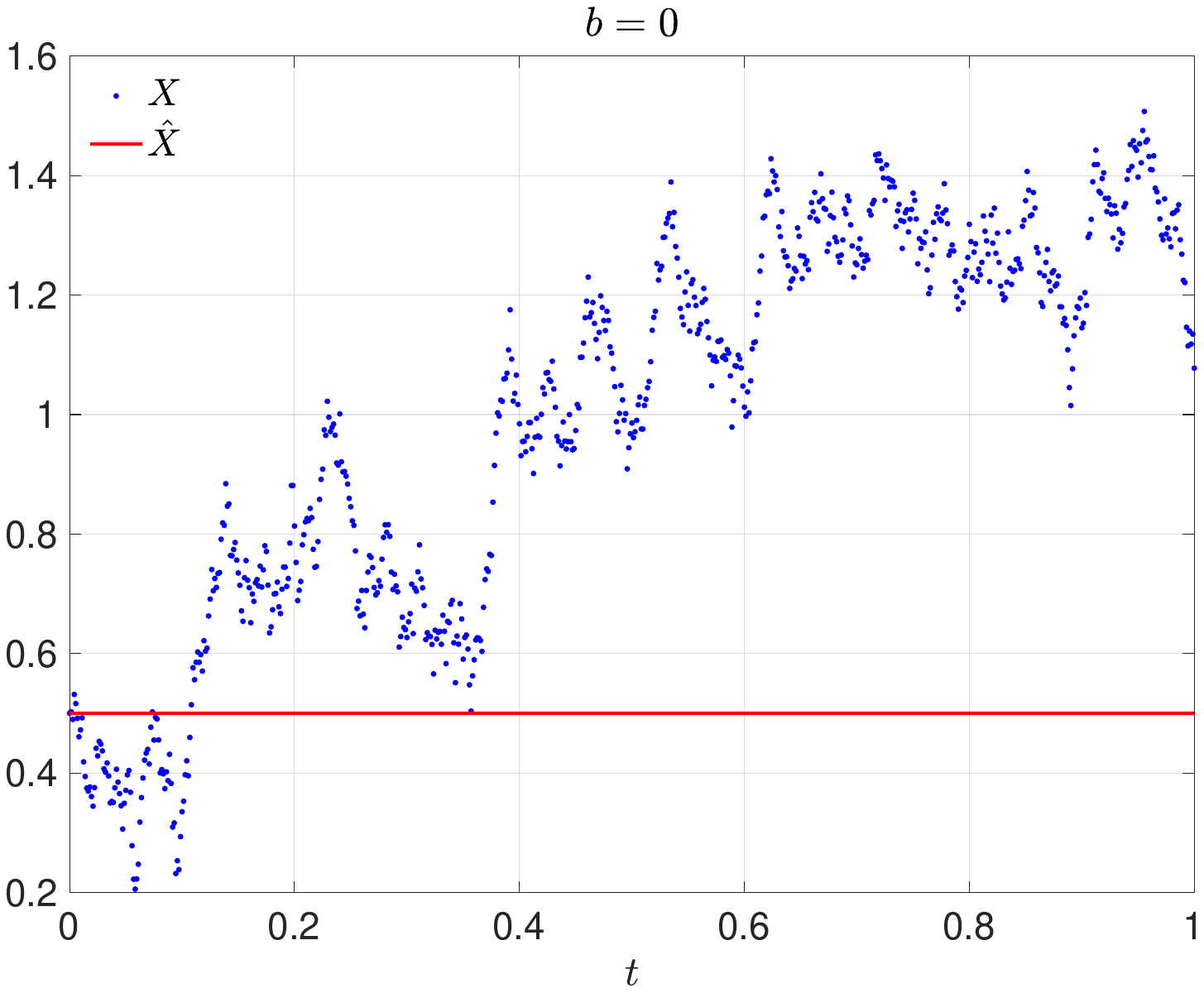} \\
(c) & (d)   \\
\includegraphics[width=0.45\textwidth]{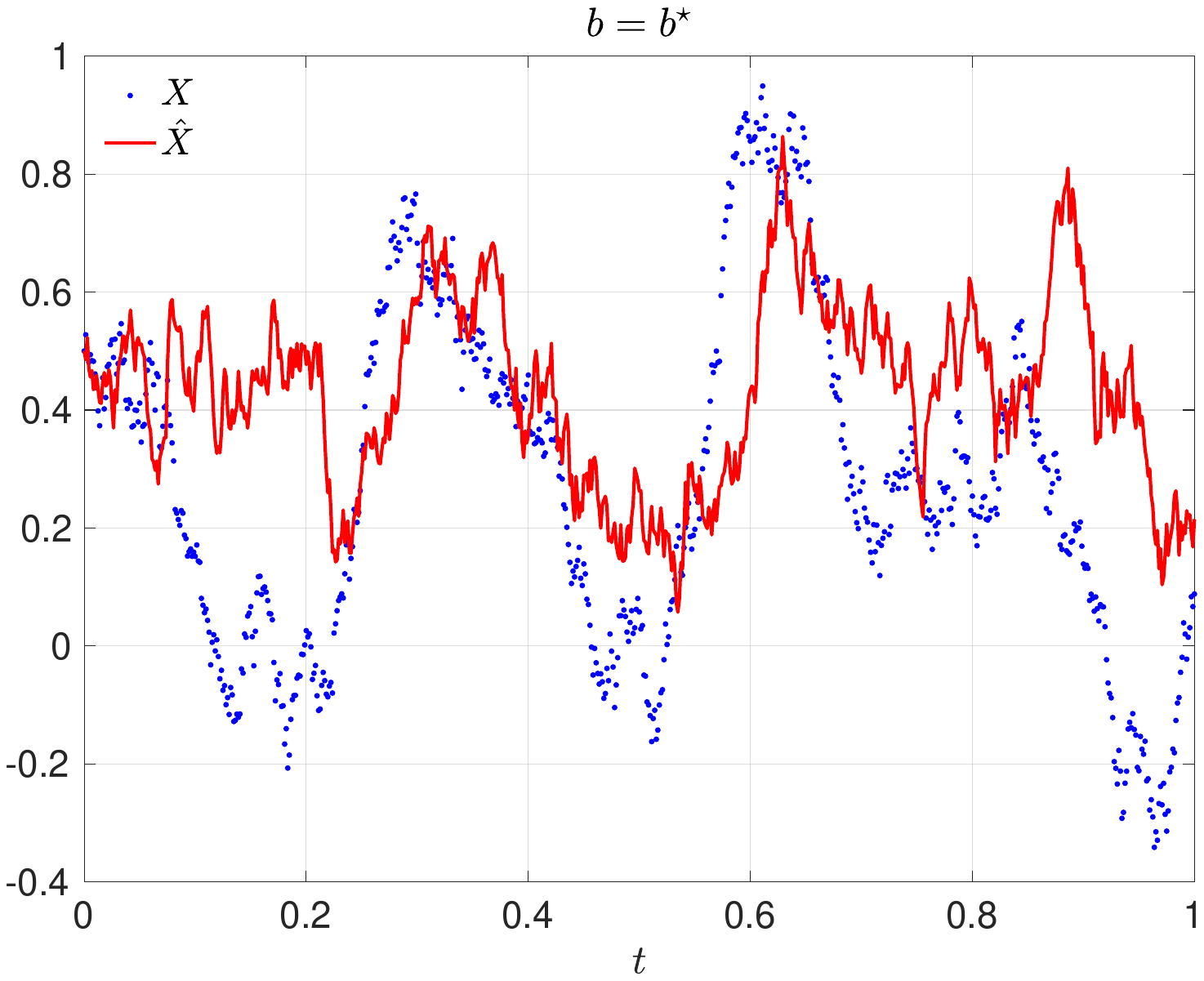}  &
\includegraphics[width=0.45\textwidth]{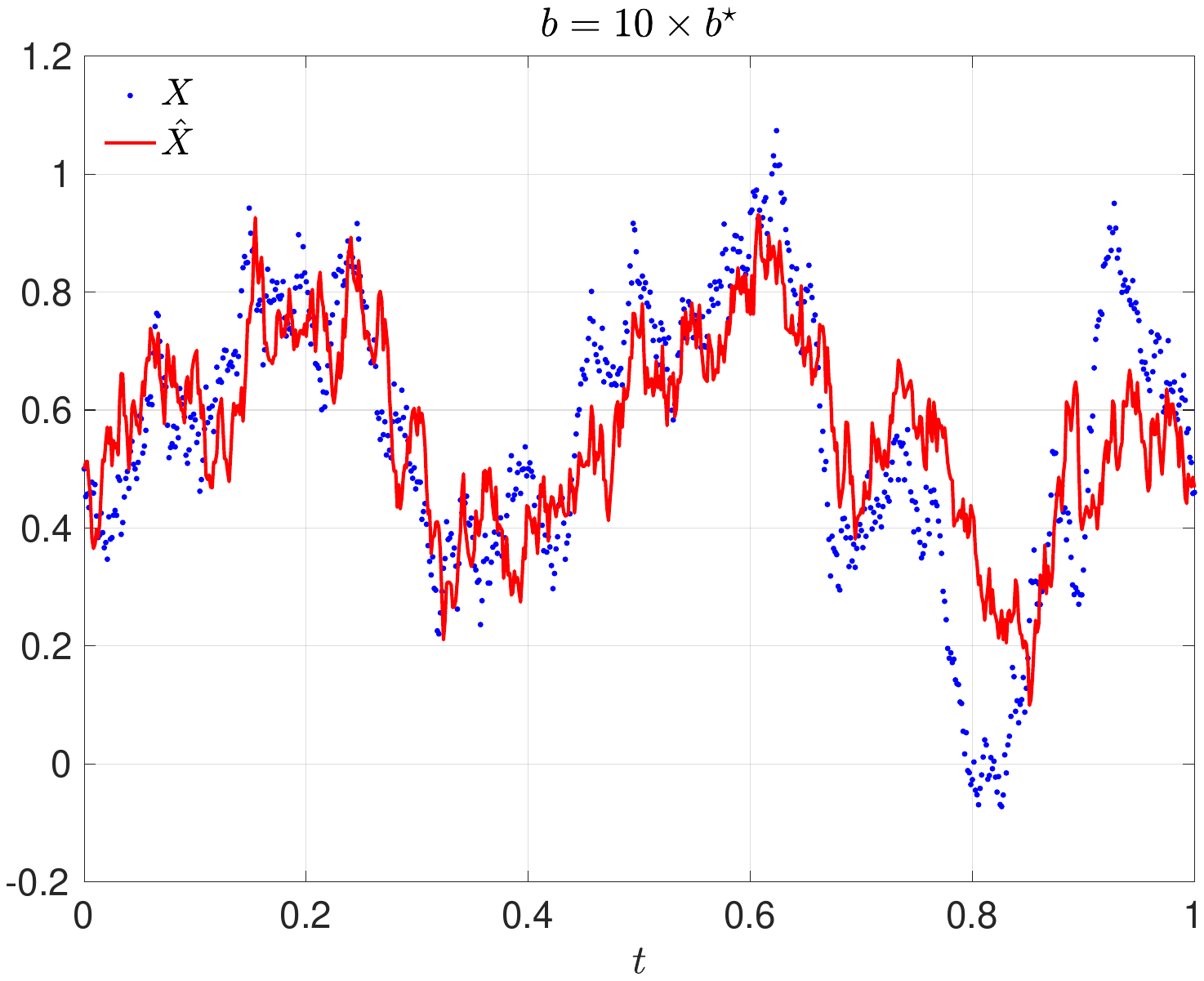} \\
\end{tabular}
\caption{ {\small  (a) The variance $\sigma^2_\infty(b)$ of the stationary law of the consumption process $X$. Three trajectories of the consumption process $X$ and its posterior estimation $\widehat X$ for (b) $b=0$ (c) $b= 5.5$ (d) $b = 55$. Parameters: $\kappa=0.5$, $\ell=1$, $u_0=100$, $\gamma = 1$, $p_0=50$, $p_1=100$; $g_1=400$, $h=0.5$.}}
\label{fig:DR}
\end{center}
\end{figure}

\subsection{Carbon footprint accounting rules}\label{ssec:res-carbon}
We focus here on the second example, introduced in Section \ref{sec:carbon}. The parameter identification (example-general theory) and the verification of the standing assumptions in this case is provided in Appendix \ref{app_ex_2}.
A direct application of the results of the general case leads to the expression of the optimal control of the representative firm:

\hs

For a given value of $b$ and a given value of $q$, the optimal dynamics of the state $X ^\ast$ and its estimate $\widehat X^\ast$ are given by
\begin{align*}
&\D \widehat X_t   = - \beta \Big(\widehat X_t -  \hat\ell(q)\Big) \D t + b P_b(t) \D I_t, \quad
\D X_t  = \big[ \kappa(\widehat X_t - X_t) + \beta(\hat\ell(q) - \hat X_t)\big] \D t + \D W_t
\end{align*}
where 
\begin{align*}  
&\hat \ell(q) :=  \frac{\ell \kappa^2 + 2\gamma \bar q}{\kappa^2 + 2\gamma\bar\lambda},\quad 
    \bar\lambda := \lambda_a + \lambda_q, \quad
    \bar q := \lambda_a a + \lambda_q q, \quad
\beta := \sqrt{\kappa^2 + 2 \gamma \bar \lambda}
\end{align*}
and the optimal control is given by
  \begin{align}\label{opt_contr_energy}
&v^\ast_t =   \kappa (\widehat X_t^*- \ell) - \beta \widehat X^\ast_t + \beta \hat \ell(q).
\end{align}

 The fixed point condition on the value of the stationary target $q$ is given by $q = \hat \ell(q) - \epsilon \sigma_\infty(b)$ where the stationary variance of the carbon footprint process $X^\ast$ is given by 
\begin{align}
 & \sigma_\infty^2(b) = \frac{1}{2\kappa} \Big[1-\frac{\beta-\kappa}{\beta}\Big(\sqrt{1 + \frac{\kappa^2}{b^2}} -  \frac{\kappa}{b}\Big)^2\Big],
\end{align}
which is the same as in the smart meters application of section~\ref{ssec:res-smart}. This last function does not depend on $q$. 
Hence, we find that the equilibrium value $q_\infty^\ast(b)$ is given by
\begin{align}
q_\infty^\ast(b) := \frac{\ell\kappa^2 + 2\gamma a \lambda_a}{\kappa^2 + 2\gamma  \lambda_a}
- \epsilon \frac{\kappa^2 + 2\gamma\bar\lambda}{\kappa^2 + 2\gamma \lambda_a} \sigma_\infty(b)\end{align}
Thus, the stationary mean $m^\ast_\infty(b)$ of the process $X$ is given by $m^\ast_\infty(b) = \hat \ell(q^\ast_\infty(b)) = q^\ast_\infty(b) + \epsilon \sigma_\infty(b)$, and so
\begin{align}
	m^\ast_\infty(b)  = \frac{\kappa^2\ell + 2\gamma a \lambda_a}{\kappa^2 + 2\gamma \lambda_a } - \epsilon  \frac{2\gamma\lambda_q}{\kappa^2 + 2\gamma \lambda_a} \sigma_\infty(b)
	=: m_0 - \epsilon m_1 \sigma_\infty(b).
\end{align}

Noticing that $\mathbb E[(X^\ast_\infty)^2]= (m_\infty^\ast)^2 + \sigma_\infty^2$, the optimisation problem of the Sender can be reformulated as
\begin{align*}
	\inf_{z \in (\underline\sigma, \bar\sigma]} G(z) + \frac 12 c \big[ (m_0 - \epsilon m_1 z)^2 + z^2\big],
\end{align*}
where we exploit the notation $z := \sigma_\infty(b)$, $\underline\sigma$ and $\overline\sigma$ have been defined in~\eqref{eq:sigmalim}, and where the function $G$ measures the cost to reduce the standard deviation of the stationary law of $X$. A possible form that leads to closed-form expression for the optimal reduction is
\begin{align*}
G(z) = -  \frac{\eta}{\Delta \sigma} \ln\frac{z - \underline\sigma}{\Delta\sigma}, \quad z \in (\underline\sigma, \bar\sigma], \quad \Delta\sigma = \overline\sigma - \underline\sigma,
\end{align*}
where $\eta$ is a nonnegative parameter. This form exhibits the same type of properties as the cost function~$H$ used in the smart meters application of section~\ref{ssec:res-smart}. It is a nonincreasing, strictly convex function that satisfies $G(\overline\sigma) = 0$ and $\lim_{z\to\underline\sigma} G(z) = + \infty$. As in the smart meters application, it ensures that it is infinitely costly to reach perfect truth.

For an interior solution to exist, it is necessary and sufficient that the marginal benefit exceeds the marginal cost at $z =\bar\sigma$, i.e.,
\begin{align*}
 \eta  \leq c (\Delta\sigma)^2 \big[  (1 + \epsilon^2 m_1^2) \bar\sigma -\epsilon m_1m_0\big].
\end{align*}
If this condition is not satisfied, it is optimal to set the standard deviation to its level $\bar\sigma$, i.e., to provide no information. In the other case, the interior solution is given by
\begin{align}\label{eq:sigmaCF}
\sigma^\ast_\infty := \frac12 \Bigg[ 
 \sqrt{ \Big(\frac{\epsilon m_0 m_1}{1+\epsilon^2 m_1^2} + \underline\sigma\Big)^2 + 4 \frac{\delta - \epsilon   m_0 m_1 \underline\sigma}{1+\epsilon^2 m_1^2}} + \frac{\epsilon m_0 m_1}{1+\epsilon^2 m_1^2} + \underline\sigma  \Bigg], \quad
\text{with} \quad
\delta := \frac{\eta}{c \Delta\sigma}.
\end{align}

\begin{figure}[tbh!]
\begin{center}
\begin{tabular}{c c}
(a) & (b)   \\
\includegraphics[width=0.45\textwidth]{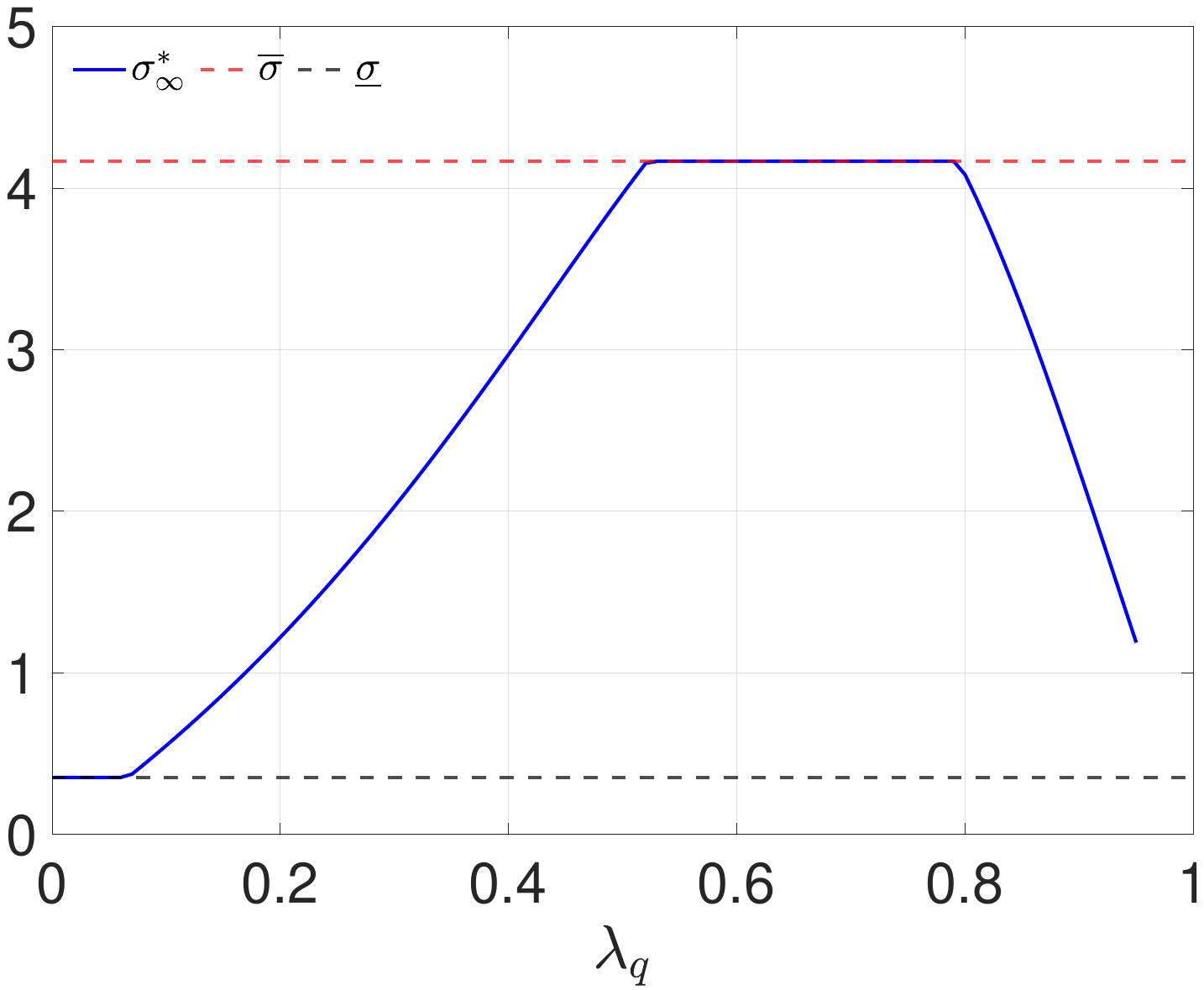} &
\includegraphics[width=0.45\textwidth]{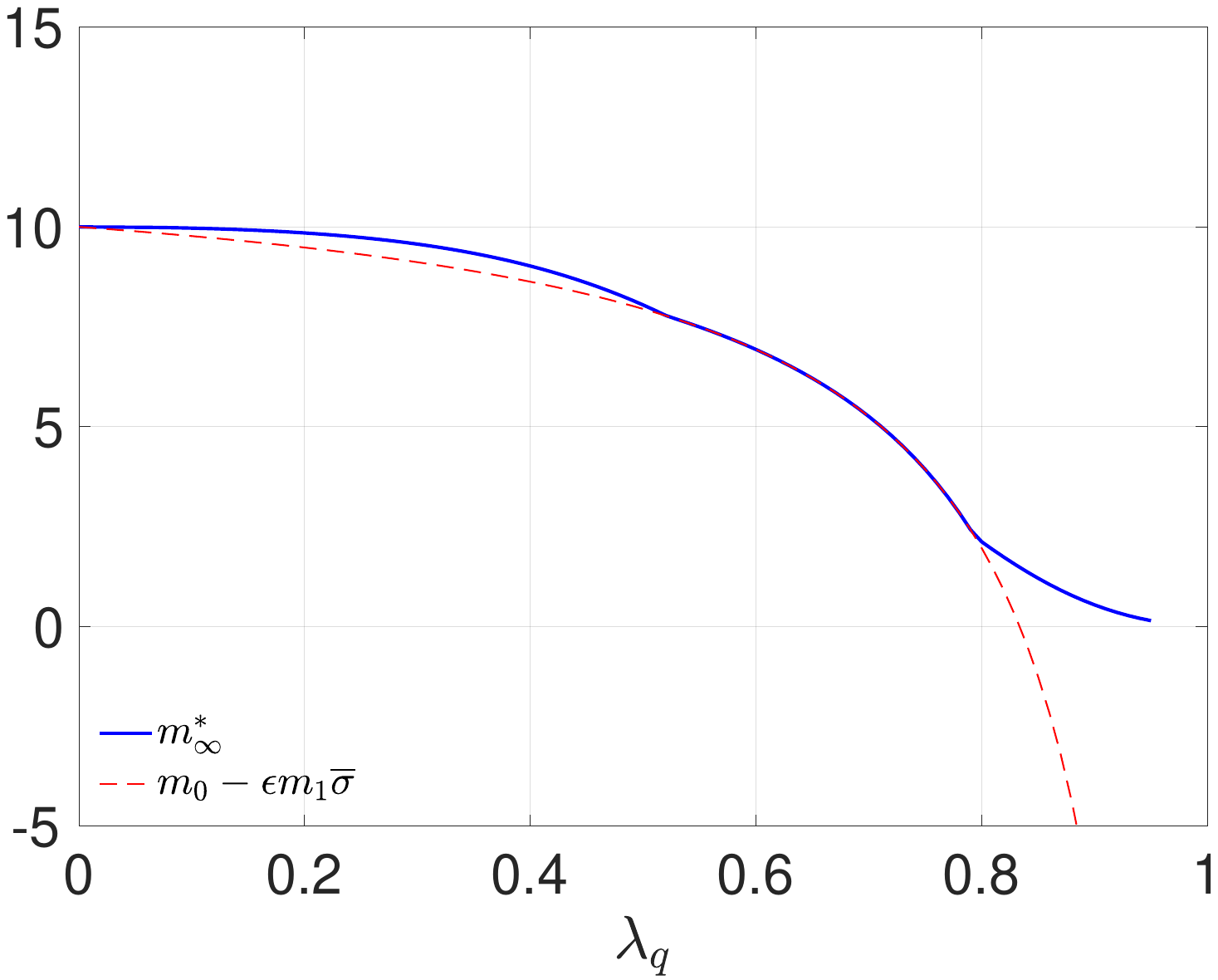} \\
(c) & (d)  \\
\includegraphics[width=0.45\textwidth]{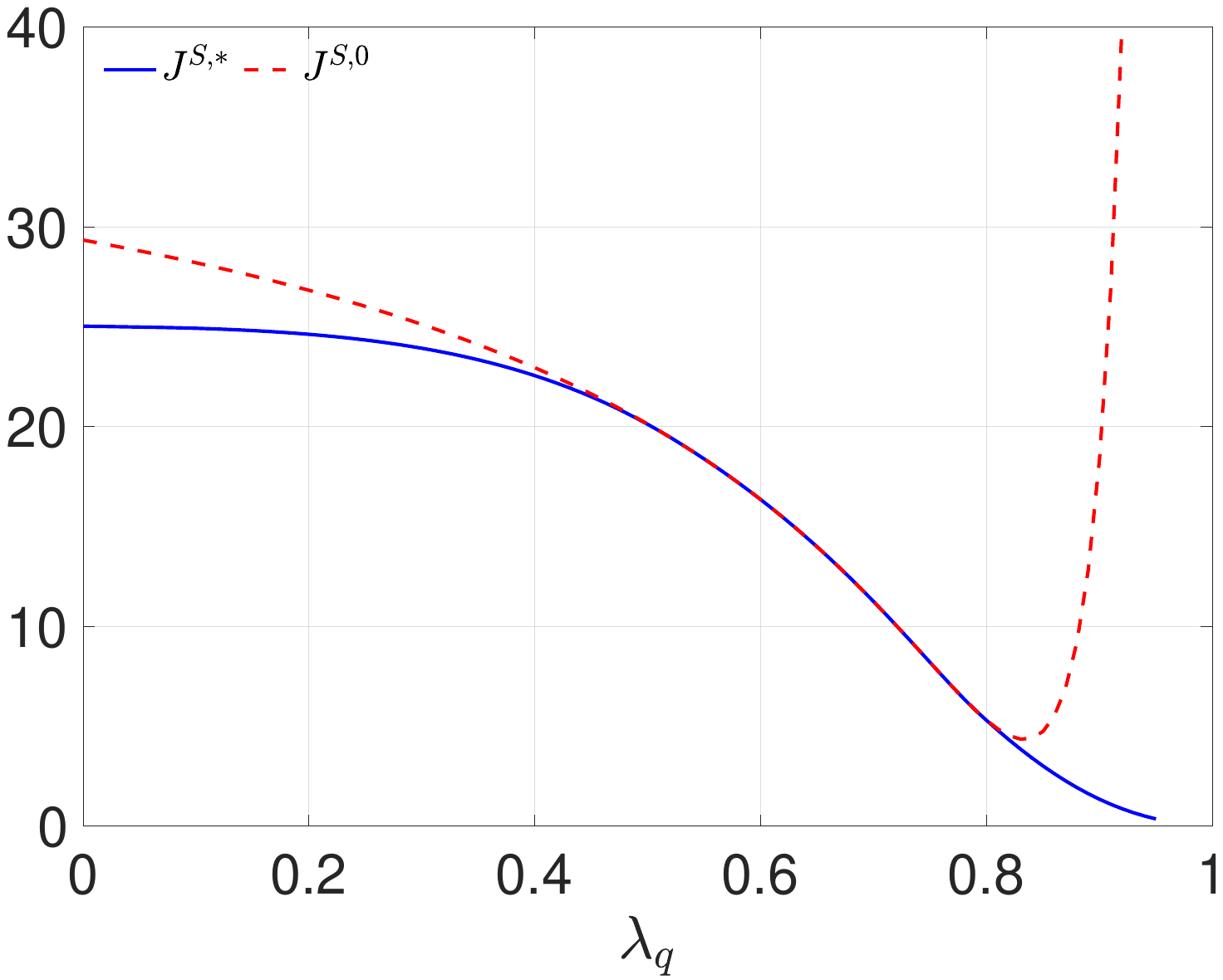} &
\includegraphics[width=0.45\textwidth]{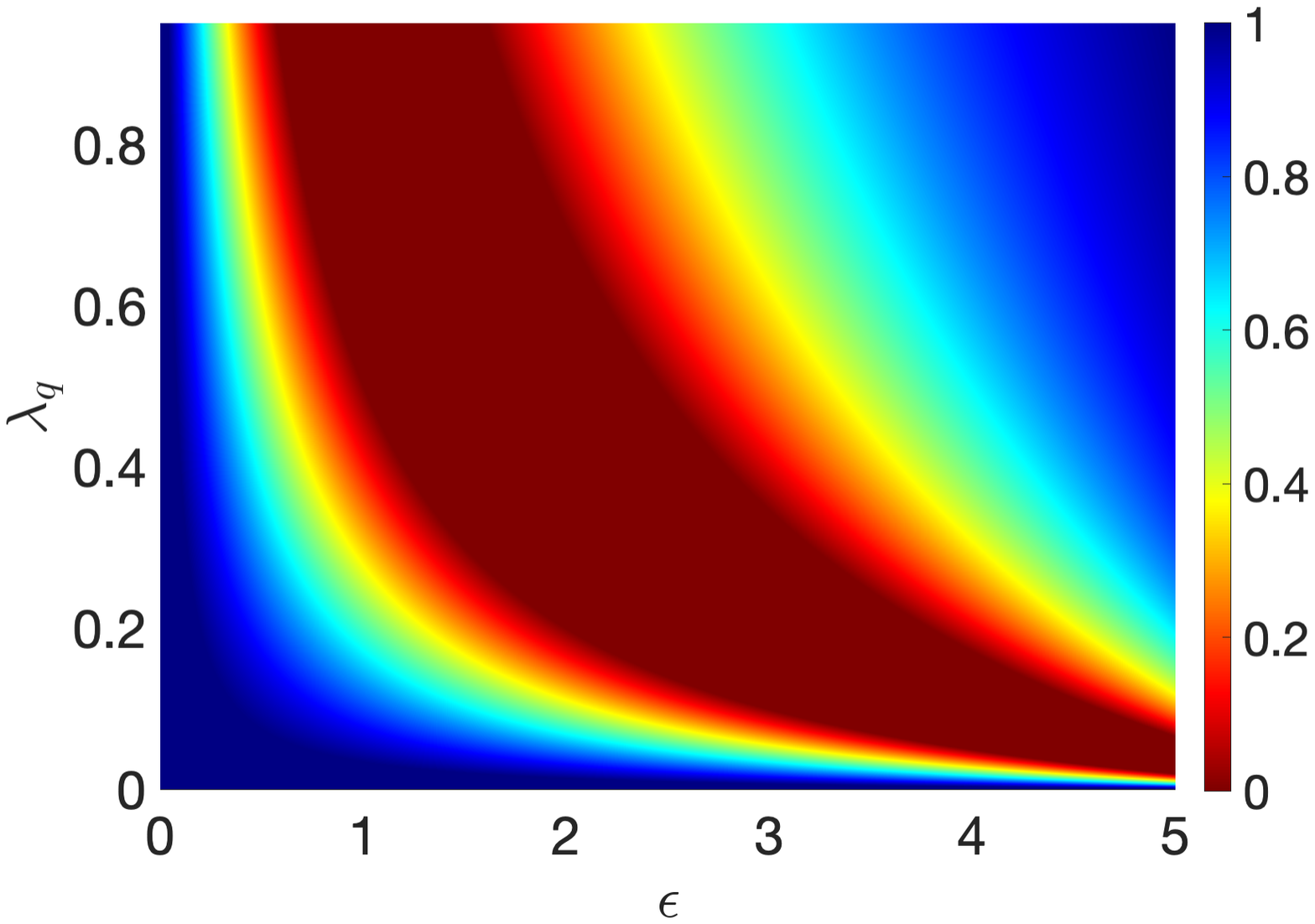} \\
\end{tabular}
\caption{ {\small Case without information cost. As a function of the weight $\lambda_q$ of the best-in-class target (a) Optimal choice of standard deviation $\sigma^\ast_\infty$ of $X^\ast_\infty$ (b) Optimal $m^\ast_\infty$ of $X^\ast_\infty$ (c) Optimal value of the cost function of the Regulator (d) Scaled nformation map ie ratio $(\overline{\sigma}-\sigma^\ast_\infty)/\Delta\sigma$ as a function of both $\lambda_q$ and $\epsilon$.  Parameters: $\ell=a=10$, $\gamma = 1$, $c=0.5$, $\kappa=0.12$, $\epsilon=0.5$.} }
\label{fig:CFP}
\end{center}
\end{figure}

\medskip

At this stage, some remarks can be done. First, note that ,  for a given level of information as measured by the standard deviation $\sigma^\ast_\infty$ of the carbon footprint process, the expected carbon footprint of the representative firm is a decreasing function of the best-in-class intensity parameter $\epsilon$. Thus, the more the firms are engaged in an emulation process in reducing their carbon footprint, the more the average carbon footprint does reduce. This is the phenomenom that information provision is supposed to reinforce or at least to help in its implementation. But, we already observed that $\sigma^\ast_\infty(b)$ is a decreasing function of the accounting norm stringency as measured by the regulator's control $b$. Thus, when the regulator provides more information to firms engaged in a best-in-class emulation process in carbon footprint reduction, the first effet of information provision is {\em to increase the average carbon footprint} of the representative firm. Indeed, information helps firms to coordinate in a process to reduce the effort of reaching a below average carbon footprint. 

Second, note that the objective of the regulator is to miminise $\EE[(X_\infty^\ast)^2] = (m_0 - \epsilon m_1 \sigma^\ast_\infty)^2 + (\sigma^\ast_\infty)^2$ where $\sigma^\ast_\infty \in (\underline{\sigma}, \overline\sigma]$. Thus, when $\epsilon>0$, the regulator has to deal with a tradeoff where information reduces the variance $(\sigma_\infty^\ast)^2$ of the firms carbon footprint but increases its average $m_\infty^\ast$. When, there is no best-in-class target (ie $\epsilon =0$ or $\lambda_q=0$), only the existence of an information cost function $G(z)$ prevents the regulator from sending full information and taking $\sigma_\infty^\ast = \underline{\sigma}$. This point is obtained directly in relation~\eqref{eq:sigmaCF} by taking $\epsilon=0$.  But as soon as firms are engaged in such an emulation process, even in the absence of information cost, the regulator may choose to deviate from sending full information on the state process. The Figure~\ref{fig:CFP} illustrates this phenomenon.

\medskip

Figure~\ref{fig:CFP}~(a), (b), (c) and (d) display $\sigma_\infty^\ast$, $m_\infty^\ast$ and the cost function of the regulator as a function of the best-in-class parameter $\epsilon$ in a situation where the is no information cost ($\eta=0$). We observe in picture~(a) that for low values of $\epsilon$, it is optimal to send full information: $\sigma^\ast_\infty$ is set at its lowest possible value. But as soon as $\epsilon$ excesses a thresold, the regulator restrains information. As $\epsilon$ increases, the regulator decreases information until it reaches the largest possible value of $\sigma^\ast_\infty$. At this stage, we observe in pictures~(b) and (c), that the average $m^\ast_\infty$ has reached its minimum possible value and the regulator's cost is equal to the case of no information provision, which is also the largest possible cost she can get.

But, more spectacular, we observe that when $\epsilon$ continues to increase, it reaches a second threshold where it is now again optimal to send information. Indeed, picture~(b) shows that if the regulator would keep sending no information, the average value $m^\ast_\infty$ would become highly negative. But, because the regulator's criterion penalises also large negative values (as if she would not believe in carbon compensation), providing information helps the firms to maintain their average carbon fooprint close to zero.

The picture~(d) of Figure~\ref{fig:CFP} provides the information map of the regulator, ie the ratio $(\overline\sigma-\sigma^\ast_\infty)/\Delta\sigma \in [0,1]$ as a function of both $\epsilon$ and $\lambda_q$.  The blue zones correspond to a ratio of $1$ while the dark red zones correspond to a ratio of $0$. It shows that the set of points where it is optimal for the regulator to send full information on the firms carbon footprint are not connex. They are separated by an information dessert where information is detrimental to a reduction of carbon footprint.

\medskip

As a result, our model shows that only information cost can justify deviation from sending full information when there is no strategic interaction between firm willing to reduce their carbon footprint. But, in the presence of strategic interaction in the form of a best-in-class emulation process, information can also be used by strategic firms to coordinate on a collectively higher level of carbon footprint to look individually cleaner. And, thus, even in the absence of information cost, it may be optimal to blur information available to firms.

\section{Conclusion}\label{sec:conclusion}

In the context of continuous-time dynamics, the economic literature on Bayesian persuasion has focused on time-dependent strategies employed by the Sender for the purpose of information provision. Moreover, it has been assumed that the Receiver in question is myopic, thereby enabling the definition of an HJB equation for the Sender. In our paper, we adopt an alternative approach. We have developed a dynamic framework for the provision of information in which the Sender defines the efficiency characteristics of the information device on a one-time basis. This approach ensures that the Sender is bound by her decision. The applications presented in the context of electricity consumption flexibility and carbon footprint are intended to demonstrate that, despite the static nature of the Sender's decision-making process, our framework allows for a diverse range of economic interpretations.  

The framing of continuous-time Bayesian persuasion in a filtering framework with an ergodic criterion allows for the development of tractable applications with a richer set of possible state dynamics. This paper focuses on Kalman-Bucy filtering, a topic that has not been extensively exploited in the existing literature. The general infinite-dimensional filtering equation makes it unlikely to have any application yielding tractable computations and closed-form expressions as presented in our paper. Nevertheless, future developments of the present setting may consider the possibility of device characteristics changing over time in accordance with a finite-state Markov chain. In such a scenario, both the Receiver and the Sender could be conceptualised as agents with rational anticipation, resulting in a stochastic dynamic leader-follower problem. However, the issue of commitment to an information provision strategy is likely to resurface, necessitating the investigation of methods to reinforce the Sender's commitment. 

\paragraph{Acknowledgements:} 
René Aïd thanks the financial support of the {\em Finance and Sustainable Development EDF-CA CIB Chair}, the {\em Finance for Energy Market} Research Initiative, the French ANR PEPR Math-Vives project MIRTE ANR-23-EXMA-0011, the Department of Mathematics of the University of Padova and the grant 6301-2 from the Indo-French Centre for the Promotion of Advanced Research. He is particularly grateful to Ankur Kulkarni (IIT Bombay) for introducing him to the use of information signals in the shaping of agents' behaviour and to Françoise Forges for her explanations on the differences with cheap talk.\smallskip

Ofelia Bonesini acknowledges financial support from Engineering and Physical Sciences Research
Council [Grant EP/T032146/1].\smallskip

Giorgia Callegaro and Luciano Campi acknowledge financial support under the National Recovery and Resilience Plan (NRRP), Mission 4, Component 2, Investment 1.1, Call for tender No. 1409 published on 14.9.2022 by the Italian Ministry of University and Research (MUR), funded by the European Union – NextGenerationEU – Project Title: Probabilistic Methods for Energy Transition -- CUP G53D23006840001 - Grant Assignment Decree No. 1379 adopted on 01/09/2023 by MUR. \smallskip

The authors thanks for their valuable comments in improving this manuscript the attendees of the Bachelier Congress 2024 edition, the Hammamet Conference on Stochastic Control and Games, the Banff BIRS workshop on ``Modeling, Learning and Understanding'', the Online Seminar of International Bachelier Society, the Oxford seminar of applied mathematics, the Bachelier seminar at IHP, Paris, and in particular Agostino Capponi, Robert Crowell, Mike Ludkovski, Marcel Nutz, Fenghui Yu and Yufei Zhang.

\appendix

\section{Verification of the standing assumptions in the applications}

\subsection{The informative value of smart meters}\label{app_ex_1}

We start with the first example, introduced in Section \ref{sec:ex_elec}. The following mapping between model parameters holds: $d_W = d_B = r=1$ and
$$
A_{\rm x} = - \kappa; \quad  B_{\rm x}  = 1 ; \quad  c_{\rm x} =  \kappa \ell; \quad b= b/\sigma.
$$
The function in the Receiver's objective functional is $f(x,v) =  (p_1 + u_0) x^2 + ( p_0 - 2 \ell u_0) x+ u_0 \ell^2 + \frac{1}{2\gamma} v^2$, with $p_0, p_1, u_0, \ell,\gamma >0$. Hence, in this example, we have  (see Equation \eqref{eq:f}): $F_0 = u_0 \ell^2, F_1 = p_0 - 2 \ell u_0, F_2 = p_1 + u_0, C_0=0, C_1= \frac{1}{2 \gamma}$. 
The function in the Sender functional is $g(x,v) = g_0 x + \frac12 g_1 x^2 - (p_0 + p_1 x) x$, with $g_0, g_1 >0$.

Assumption \ref{ass_coeff} $i)$ (ensuring existence of $P_b(\infty)$ and so filter stability at $+\infty$) is satisfied, since, exploiting Definition \ref{def:stab_obs} with $n=m=1$, we have that $(- \kappa, b)$ is \emph{stabilisable}. Indeed, being $b$ finite, there always exists a real number $c$ such that $-\kappa + b c <0$.

Assumption \ref{ass_coeff} $ii) - v)$ (ensuring existence of the optimal $G_2$) is also satisfied, as we have: $ii) \ F_2 = p_1 + u_0 \ge 0$; $iii) \ B_{\rm x} C_2^{-1} B_{\rm x}^\top = 2 \gamma \ge 0$;
$iv)\  (-A_{\rm x}, B_{\rm x} C_2^{-1} B_{\rm x}^\top) = (\kappa, 2 \gamma)$ stabilisable; $v) \  (-A_{\rm x}^\top,F_2^\top) = (\kappa, p_1 + u_0) $ stabilisable.
So, existence of $G_2$ (and $G_1$, see below) is granted and we have
$$
G_2 = \frac{\beta - \kappa}{2 \gamma} >0, \quad G_1 = \frac{1}{\beta} \left( p_0 - 2 \ell u_0 + \frac{\kappa \ell}{\gamma} (\beta - \kappa) \right),
$$
with $\beta = \sqrt{\kappa^2 + 2 \gamma (p_1 + u_0)} > \kappa >0$. Notice that $G_2$ solves a quadratic equation, admitting two solutions and we choose the one for which  Assumption \ref{ass_Theta1} holds,  since $\Theta_1 := - \kappa - 2 \gamma G_2 = - \beta $ is striclty negative.
Moreover, $G_1$ exists, given that $\beta \neq 0$.
Finally, Assumption \ref{ass_stab_cov} is also satisfied, since the matrix $\Theta$:
\begin{align*}  \Theta(t) & =
\begin{pmatrix}
        -\kappa  & \kappa - \beta\\
        P_b(t) b^2 & -\beta -  P_b(t)b^2
    \end{pmatrix}
    \end{align*} 
is stable. This comes from the fact that a $2\times2$ real matrix, $A$, is stable if and only if $\textrm{Tr}(A) <0$ and $\textrm{det}(A)>0$, which is the case here. Moreover, $(\Xi, \Theta^\top )$ is observable (recall Definition \ref{def:stab_obs}), with $\Xi(t) : =
    \begin{pmatrix}
        1 &  0\\
         0 & P_b(t)b 
    \end{pmatrix}$, since $\Xi$ has rank equal to two for $b \neq 0$.

\subsection{Carbon footprint accounting rules}\label{app_ex_2}

We focus here on the second example, introduced in Section \ref{sec:carbon}. The following mapping between model parameters holds: $d_W = d_B = r=1$ and
$$
A_{\rm x} = - \kappa; \quad  B_{\rm x}  = 1 ; \quad  c_{\rm x} =  \kappa \ell; \quad b= b/\sigma.
$$
The function in the Receiver's objective functional is $f(x,v) = (\lambda_q + \lambda_a)x^2  -2 ( a \lambda_a + q \lambda_q) x +\lambda_q q^2 + \lambda_a a^2+ \frac{1}{2\gamma} {v^2}$, with $\lambda_a, \lambda_q,\gamma >0$. Hence, we have $F_0 = \lambda_q q^2 + \lambda_a a^2, F_1 = -2 a \lambda_a - 2 q \lambda_q , F_2 = \lambda_q + \lambda_a, C_1=0, C_2= \frac{1}{2 \gamma}$. 
The function in the Sender functional is $g(x,v) = \frac12 c x^2$, with $c >0$.

As in the previous example, Assumption \ref{ass_coeff} $i)$ is satisfied, since $(- \kappa, b)$ is \emph{stabilisable}. 
Assumption \ref{ass_coeff} $ii)-v)$ holds, since: $ii) \ F_2 = \lambda_q + \lambda_a  \ge 0$; $iii) \  B_{\rm x} C_2^{-1} B_{\rm x}^\top = 2 \gamma \ge 0$;
$iv)\  (-A_{\rm x}, B_{\rm x} C_2^{-1} B_{\rm x}^\top) = (\kappa, 2 \gamma)$ stabilisable; $v) \  (-A_{\rm x}^\top,F_2^\top) = (\kappa, \lambda_q + \lambda_a) $ stabilisable.
So, existence of $G_2$ (and, as a consequence, of $G_1$, see below) is granted and we have:
$$
G_2 = \frac{\beta - \kappa}{2 \gamma} >0, \quad G_1 = \frac{1}{\beta \gamma} \left( \kappa \ell (\beta - \kappa) - 2 \gamma (a \lambda_a +  q \lambda_q) \right),
$$
where $\beta = \sqrt{\kappa^2 + 2 \gamma ( \lambda_q + \lambda_a)} > \kappa$. As in the previous example, $G_2$ solves a quadratic equation, admitting two solutions. We choose the one for which  Assumption \ref{ass_Theta1} holds,  since $\Theta_1 := - \kappa - 2 \gamma G_2 = - \beta $ is striclty negative.
Being $\beta \neq 0$, $G_1$ is well defined.
Finally, Assumption \ref{ass_stab_cov} is also satisfied, since the matrix
\begin{align*}  \Theta(t) & =
\begin{pmatrix}
        -\kappa  & \kappa - \beta\\
        P_b(t) b^2 & -\beta -  P_b(t)b^2
    \end{pmatrix}
    \end{align*} 
is stable. To see this we use once more the fact that $\textrm{Tr}(\Theta) <0$ and $\textrm{det}(\Theta)>0$. 
Moreover, $(\Xi, \Theta^\top )$ is observable (recall Def. \ref{def:stab_obs}), with $\Xi(t) : =
    \begin{pmatrix}
        1 &  0\\
         0 & P_b(t)b 
    \end{pmatrix}$, since $\Xi$ has full rank for $b \neq 0$.

\hs

Moreover, we provide hereafter a detailed version of the computations that leads to the solution to illustrate the easyness of the method.

\hs

We have $dX_t = \kappa (\ell - X_t) dt + v_t dt  + dW_t$ and $dM_t = b X_t dt + dB_t$ and the Receiver's rate of profit is $f(x,v) := c(v) + \lambda (x-q)^2 + \lambda_a (x-a)^2$ with $c(v) := \frac12 v^2/\gamma$. The estimated process dynamics is $dY_t = \kappa (\ell - Y_t) dt + v_t dt+ b P_b(t) dI_t$ where $I$ is the innovation process.  The variance of the estimator $P_b(t)$ satisfies
$\dot P_b = - 2 \kappa P_b - b^2 P_b^2 + 1$, with $P_b(0) = 0$, and we are only interested in its stationary value given by $- 2 \kappa P_\infty - b^2 P_\infty^2 + 1 = 0$. The value function $V$ of the Receiver satisfies the HJB equation
\begin{align*}
\zeta = \frac12 P_\infty^2 V_{yy} + \inf_{v} \big\{ k(\ell - y) V_y  + v V_y + f(y,v) \}, \quad
P_\infty = \frac1{b^2} \sqrt{\kappa^2 + b^2} - \frac{\kappa}{b^2}.
\end{align*}
From which we get $v^\ast = - \gamma V_y$ and thus
\begin{align*}
\zeta = \frac12 P_\infty^2 V_{yy} + \kappa (\ell - y) V_y  - \frac12 \gamma V_y^2 + \lambda_a (y -a)^2 + \lambda_q (y-q)^2.
\end{align*}
Assuming $V(y) = A y^2 + B y + C$ we get
\begin{align*}
-2\kappa A  - 2 \gamma A^2 + \bar\lambda = 0, \quad
2 \kappa \ell A - \kappa B -  2\gamma AB -  2\bar q  = 0, \quad \bar\lambda := \lambda_a + \lambda_q, \quad
\bar q := a \lambda_a + q \lambda_q
\end{align*}
Hence, $V$ being convex we pick $A = \frac{\beta - \kappa}{2\gamma}$ $>0$, with $\beta := \sqrt{\kappa^2 + 2\gamma\bar\lambda}$. And we get 
\begin{align*}
B =  \frac{1}{\gamma \beta} \Big( \kappa\ell(\beta-\kappa) - 2\gamma \bar q\Big).
\end{align*}
Hence, the optimal control is
\begin{align*}
v^\ast (y)= - 2 \gamma A y - \gamma B = (\kappa - \beta) y - \frac{1}{\beta}\Big( \kappa\ell(\beta-\kappa)-2 \gamma\bar q\Big)
\end{align*}
Thus the dynamics of the estimated process is
\begin{align*}
d\hat X_t  & = \kappa (\ell - \hat X_t) dt + (\kappa - \beta) \hat X_t dt - \frac{1}{\beta}\Big( \kappa\ell(\beta-\kappa)-2 \gamma\bar q\Big) dt + b P_b(t) dI_t, \\
& = -\beta \hat X_t dt + \frac{\ell \kappa^2 + 2\gamma \bar q}{\beta} dt + b P_b(t) dI_t 
= -\beta \Big(\hat X_t -  \frac{\ell \kappa^2 + 2\gamma \bar q}{\beta^2}\Big) dt 
+ b P_b(t) dI_t
\end{align*}
Hence, the dynamics of the estimated process is
\begin{align*}
d\hat X_t  & = \beta\big(\hat\ell(q) - \hat X_t\big) dt + b P_b(t) dI_t, \quad \hat\ell(q) := \frac{\ell \kappa^2 + 2\gamma \bar q}{\kappa^2 + 2\gamma\bar\lambda}
\end{align*}
Using the fact that $dI_t = dM_t - b\hat X_t dt$, we have also
\begin{align*}
d\hat X_t  & = \big[ \beta\big(\hat\ell(q) - \hat X_t\big) + b^2 P_b(t)(X_t - \hat X_t) \big]dt + b P_b(t) dB_t, 
\end{align*}

Besides, the dynamics of the state process $X$ is given by
\begin{align*}
dX_t  & = \big[ \kappa(\ell - X_t) + (\kappa - \beta) \hat X_t - \kappa\ell + \beta \hat\ell(q) \big] dt + dW_t \\
& =  \big[- \kappa X_t + (\kappa - \beta) \hat X_t  + \beta \hat\ell(q) \big] dt + dW_t \\
& = \big[ \kappa(\hat X_t - X_t) + \beta(\hat\ell(q) - \hat X_t)\big] dt + dW_t
\end{align*}

We define $m_t:= \mathbb E (X_t^*)$ and $U_t := \mathbb E [{(X_t^*)}^2] $, $Y_t := \mathbb E [{(\widehat X_t^*)}^2] $ and $Z_t := \mathbb E [ \widehat X_t^* X_t^*] $ and we will need the following:
\begin{align*}
d m_t & = \beta( \hat\ell(q) - m_t) dt, \quad
d U_t  =  \big[2 (\kappa - \beta) Z_t - 2 \kappa U_t + 2 \beta m_t  \hat\ell(q) +1 \big] dt, \\
d Y_t & = \big[-  2 Y_t (\beta + b^2 P_b(t)) + 2 b^2 P_b(t) Z_t + 2 \beta  \hat \ell(q) m_t + b^2 P_b^2(t)\big] dt, \\
dZ_t &  = \big[ \kappa(Y_t - Z_t) + \beta(\hat\ell(q) m_t - Y_t)  + \beta(\hat\ell(q)m_t - Z_t) 
+ b^2P_b(t)(U_t  - Z_t) \big] dt. \\
dZ_t  & = \big[ (\kappa-\beta) Y_t - (\kappa+\beta + b^2 P_b(t)) Z_t 
+ b^2 P_b(t) U_t + 2 \beta\hat\ell(q) m_t \big] dt.
\end{align*}
Clearly $m_\infty  =  \hat\ell(q)$  and the stationary solutions of the above ODEs $U_\infty, Y_\infty,Z_\infty$ solve the system:
\begin{align*}
& \big[2 (\kappa - \beta) Z_\infty - 2 \kappa U_\infty + 2 \beta  (\hat\ell(q))^2 +1 \big] =0 , \\
& \big[-  2 Y_\infty (\beta + b^2 P_b(\infty)) + 2 b^2 P_b(\infty) Z_\infty + 2 \beta  (\hat \ell(q))^2 + b^2 P_b^2(\infty)\big] =0 , \\
& \big[ \kappa(Y_\infty - Z_\infty) + \beta( (\hat\ell(q))^2 - Y_\infty)  + \beta((\hat\ell(q))^2 - Z_\infty) +b^2P_b(\infty)(U_\infty  - Z_\infty) \big]  = 0\\ 
\end{align*}
The stationary variance is:
$$
\sigma^2_\infty(b) = U_\infty- (m_\infty)^2,
$$
hence, $X^\ast_\infty$ is Gaussian with mean $\hat\ell(q)$ and variance given by $\sigma^2_\infty(b)$. 
\hs

Now, either we solve the above system or, we exploit the general formulas in Eq. \eqref{eq:ODE_m} and \eqref{eq:ODE_w}, namely
\begin{align}
    \D m_t & = 
    \left(
   - \beta   m_t + \frac{\kappa^2 \ell + 2 \gamma \bar q}{\beta}
    \right) \D t , \quad m_0 = x_0^\top \\ 
    \D w_t & = ( w_t \Theta(t)^\top + \Theta(t) w_t +\Xi(t)\Xi(t)^\top)  \D t ,  \quad w_0 = \mathbf 0,
\end{align}
with
\begin{align*}  \Theta(t) & :=
\begin{pmatrix}
        -\kappa & \kappa - \beta \\
        P_b(t) b^2 & -\beta - P_b(t)b^2
    \end{pmatrix}
    \end{align*}
    and
    \[ \Xi(t) : =
    \begin{pmatrix}
        1 & 0\\
        0 & P_b(t)b 
    \end{pmatrix}.
    \]
So, we find $m_\infty = \frac{\ell \kappa^2 + 2\gamma \bar q}{\beta^2} = \hat \ell(q)$ as above and, introducing the moments
$$
w_1 := \mathbb{V}\text{\rm ar}(X_\infty^*) , \qquad w_2 = \mathbb{C}\text{\rm ov}(X_\infty^*, {\widehat X_\infty}^*), \qquad w_3 = \mathbb{V}\text{\rm ar}({\widehat X_\infty}^*)
$$
we have that they solve the system
\begin{align*}
& - 2 \kappa w_1+ 2 (\kappa - \beta)w_2 +1 =0,\\
& P_b(\infty) b^2w_1- w_2 (\beta + \kappa + P_b(\infty) b^2) + (\kappa - \beta) w_3 =0 , \\
& 2 P_b(\infty) b^2 w_2 + P_b^2(\infty) b^2 - 2 w_3 (\beta + P_b(\infty) b^2) = 0.\\ 
\end{align*}
Recalling that $P_{b}(\infty)  = \frac{1}{b^2}\big(\sqrt{ \kappa^2 + b^2 } -\kappa\big)$, we find $w_2$
$$
w_2 = \frac{1}{2\beta} \left(  \sqrt{1 + \frac{\kappa^2}{b^2}} - \frac{\kappa}{b} \right)^2, 
$$
and then our final target $\sigma_b^2(\infty)$:
$$
\quad w_1 = \sigma_b^2(\infty)=\frac{1}{2\kappa} \left[ 1 - \frac{\beta - \kappa}{\beta}   \left(  \sqrt{1 + \frac{\kappa^2}{b^2}} - \frac{\kappa}{b} \right)^2 \right]
$$

The fixed point condition for the equilibrium $q$ is given by $q = \EE\big[ X_\infty\big] - \epsilon \sigma_\infty(b)$ with $\EE\big[ X_\infty\big] = \hat\ell(q)$. Hence we have
\begin{align*}
&q = \frac{\ell \kappa^2 + 2\gamma \bar q}{\kappa^2 + 2\gamma\bar\lambda} - \epsilon \sigma_\infty(b), \quad \text{thus} \quad
 q^\ast(b) := \frac{\ell\kappa^2 + 2\gamma a \lambda_a}{\kappa^2 + 2\gamma  \lambda_a}
- \epsilon \frac{\kappa^2 + 2\gamma\bar\lambda}{\kappa^2 + 2\gamma \lambda_a} \sigma_\infty(b).
\end{align*}
Besides, the stationary equilibrium mean $m_\infty^\ast(b)$ is given by $m_\infty^\ast(b) = \hat\ell(q^\ast(b))$. Using the fact that  $\hat\ell(q^\ast(b)) = q^\ast(b) + \epsilon \sigma_\infty(b)$, we get
\begin{align*}
m_\infty^\ast(b) &= \frac{\ell\kappa^2 + 2\gamma a \lambda_a}{\kappa^2 + 2\gamma  \lambda_a} + 
\epsilon\sigma_\infty(b)
 \Big[ 1 -  \frac{\kappa^2 + 2\gamma\bar\lambda}{\kappa^2 + 2\gamma \lambda_a}\Big], \\
m_\infty^\ast(b) & = \frac{\ell\kappa^2 + 2\gamma a \lambda_a}{\kappa^2 + 2\gamma  \lambda_a} 
- \epsilon\sigma_\infty(b)
  \frac{2\gamma \lambda_q}{\kappa^2 + 2\gamma \lambda_a} =: m_0 - \epsilon m_1 \sigma_\infty(b).
  \end{align*}
Now, using  $\mathbb E[(X^\ast_\infty)^2]= (m_\infty^\ast)^2 + \sigma_\infty^2$, the optimisation problem of the Sender can be reformulated as
\begin{align*}
\inf_{z \in (\underline\sigma, \bar\sigma]} G(z) +  \frac12 c \big[ (m_0 - \epsilon m_1 z)^2 + z^2\big], \quad
G(z) = -  \frac{\eta}{\Delta \sigma} \ln\frac{z - \underline\sigma}{\Delta\sigma}, \quad z \in (\underline\sigma, \bar\sigma], \quad \Delta\sigma = \overline\sigma - \underline\sigma,
\end{align*}
Interior solution condition reads
\begin{align*}
&0 \leq G'(\overline\sigma) + c\big[ -\epsilon m_1 (m_0 - \epsilon m_1 \overline\sigma) + \overline\sigma\big], \\
&\delta \leq   \Delta\sigma \big[(1 + \epsilon^2 m_1^2) \overline\sigma -\epsilon m_1 m_0\big]
\end{align*}
with  $\delta := \eta/(c \Delta\sigma)$.
Assuming the condition above, first-order optimality condition is
\begin{align*}
&-\frac{\eta}{\Delta\sigma(z-\underline\sigma)} + c \big[ (1+\epsilon^2 m_1^2) z -\epsilon m_1 m_0\big] = 0 \\
&- \eta + c \Delta\sigma (z-\underline\sigma) \big[ (1+\epsilon^2 m_1^2) z -\epsilon m_1 m_0 \big] = 0 \\
&- \eta + c \Delta\sigma  
 \big[ (1+\epsilon^2 m_1^2) z^2 
 - \big( \epsilon m_0 m_1 + (1+\epsilon^2 m_1^2)\underline\sigma  \big) z 
 + \epsilon m_1 m_0 \underline\sigma \big] = 0 
\end{align*}
 We have
\begin{align*}
&   (1+\epsilon^2 m_1^2) z^2 
 - \big( \epsilon m_0 m_1 + (1+\epsilon^2 m_1^2)\underline\sigma  \big) z 
 + \epsilon m_1 m_0 \underline\sigma  - \delta = 0 \\
&    z^2 
 - \Big( \frac{\epsilon m_0 m_1}{1+\epsilon^2 m_1^2} + \underline\sigma  \Big) z 
 - \frac{\delta - \epsilon m_1 m_0 \underline\sigma}{1+\epsilon^2 m_1^2} = 0 \\
 z^\ast &= \frac12 \Bigg[ \sqrt{ \Big( \frac{\epsilon m_0 m_1}{1+\epsilon^2 m_1^2} + \underline\sigma  \Big)^2 
 + 4 \frac{\delta - \epsilon m_1 m_0 \underline\sigma}{1+\epsilon^2 m_1^2} }
 + \Big( \frac{\epsilon m_0 m_1}{1+\epsilon^2 m_1^2} + \underline\sigma  \Big)\Bigg].
 \end{align*}\hfill$\Box$
 
 \hs

\appendix

\end{document}